\documentclass{amsart}
 \newcommand\Pn[1]{{\bf P}^{#1}}
 
 \newtheorem{theorem}{{\rm T\sc heorem}}[section]
 \newtheorem{lemma}[theorem]{{\rm L\sc emma}}
 \newtheorem{corollary}[theorem]{{\rm C\sc orollary}}
 \newtheorem{proposition}[theorem]{{\rm P\sc roposition}}
 \newtheorem{claim}[theorem]{Claim}
 \newtheorem{sublemma}[theorem]{Sublemma}

 \newtheorem{definition}[theorem]{{\rm D\sc efinition}}
 
 \theoremstyle{definition}
 \newtheorem{remark}[theorem]{{\rm R\sc emark}}
  
  \newtheorem*{acks}{Acknowledgments}

 \newtheorem{examp}[theorem]{Example}
 
 \newcommand\C{{\bf C}}
 \newcommand\CC{{\bf C}}
 \newcommand\ZZ{{\bf Z}}
 \newcommand\QQ{{\bf Q}}
 \newcommand\GG{{\bf G}}
\newcommand\PP{{\bf P}}

\begin{document}
 \title {Variety of power sums and divisors in the moduli space of cubic fourfolds}
 \author[Kristian Ranestad]{Kristian Ranestad}
 \address{Matematisk institutt\\
         Universitetet i Oslo\\
         PO Box 1053, Blindern\\
         NO-0316 Oslo\\
         Norway}
         \thanks{K.R. partially supported by RCN project no 239015 ``Special Geometries''.}
 \author[Claire Voisin]{ Claire Voisin}
  \address{ Coll\`ege de France\\ 3 rue d'Ulm\\ 75005 Paris\\ France
  }

 \begin{abstract} We show that a cubic fourfold $F$ that is apolar to a Veronese surface has the property
  that its variety of power sums $VSP(F,10)$ is singular along a $K3$ surface of genus
  $20$ which is the variety of power sums of a sextic curve. This relates  constructions of Mukai and Iliev and
  Ranestad.
We also prove that these cubics form a divisor in the moduli space
of cubic fourfolds and  that this divisor is not a Noether-Lefschetz
divisor. We use this result to prove that there is no nontrivial
Hodge correspondence between a very general cubic and its $VSP$.

\end{abstract}
 \maketitle
 \pagestyle{myheadings}\markboth{\textsc{ Kristian
 Ranestad and Claire Voisin}}{\textsc{Varieties of power sums and cubic fourfolds}}
\section{Introduction}
For a hypersurface $F\subset \Pn n=\PP(V^*)$ defined by a homogeneous
polynomial $f\in S^dV$ of degree $d$ in $n+1$
 variables,  we define the variety
 of sums of powers
 as the Zariski closure
 \begin{eqnarray}
 \label{eq7june}VSP(F,s) = \overline{\{\{[l_1],\ldots,[ l_s]\}\in \text{Hilb}_s
 (\check{\PP^n})\mid\exists\lambda_i \in \C:
 f=\lambda_1l_1^d+\ldots+\lambda_s l_s^d\}},
 \end{eqnarray}
 in the Hilbert scheme $\text{Hilb}_s(\check{\PP^n})$,
 of the set of power sums presenting $f$ (see
 \cite{RS}).  The minimal $s$ such that  $VSP(F,s)$ is nonempty is called the {\bf rank} of $F$.  We will study these power sums
 using apolarity.
 Concretely, we can see the defining equation $f$ as the equation of  a hyperplane $H_f$ in the dual
 space $S^dV^*$, and more generally, we
 get for each $k\leq d$ a subspace $I_f^k:=[H_f:{\rm Sym}^{d-k}V^*]\subset S^kV^*$.
 \begin{definition} We say that a subscheme
 $Z\subset \check{\PP^n}$ is apolar to $f$ (or to $F=V(f)$) if $I_Z\subset
 I_f$, or, equivalently, $I_Z^d\subset I_f^d=H_f$.  We use  the term symmetrically, and also say that $f$ is apolar to $Z$ if $I_Z^d\subset I_f^d=H_f$.
\end{definition}
The relation between apolarity and power sums is given by the
following duality lemma (see \cite{IR}):

\begin{lemma}{\label{3.4}} Let $l_1,\ldots, l_s\in V$ be linear forms.
 %and let $L_i \in \Pnd n$ be
% the corresponding points in the dual space.
Then
 $f = \lambda_{1}l_1^d+\ldots+\lambda_{s}l_s^d$ for some
 $\lambda_{i}\in \C^{*}$ if and only if $Z = \{ [l_1],\ldots,[ l_s]\} \subset {\bf P}(V)$ is apolar to $F=V(f)$.\end{lemma}
 In the case $F\subset\PP^5$ is a general cubic hypersurface, the rank of $F$ is $10$ and the variety of $10$-power sums of $F$ is $4$-dimensional.
In the paper \cite{IR}, Iliev and the first author exhibited cubic
 fourfolds
 $F_{IR}(S)$ associated to $K3$ surfaces $S$ of degree $14$ obtained as
 the transverse intersection $\GG(2,6)\cap \PP_S$ of the Grassmannian
 $\GG(2,6)$ with a codimension $6$ linear space $\PP_S$ of $\PP(\bigwedge^2V_6)=\PP^{14}$ (see Section \ref{sec1} for
 the precise construction).
 On the other hand Beauville and Donagi, in \cite{BD}, associate  to such a $K3$ surface $S$
 the Pfaffian cubic $F_{BD}(S)$ which is the intersection of the
 Pfaffian cubic in $\PP(\bigwedge^2V_6^*)$ with the
 $\PP^5\subset \PP(\bigwedge^2V_6^*)$ orthogonal to $\PP_S$. The following result is proved in \cite{IR}.

  \begin{theorem}{\label{3.7}} For general $S$ as above, the variety $VSP(F_{IR}(S),10)$ is isomorphic to the family of
 secant lines to $S$, i.e. to ${\rm Hilb}_2(S)$.
 \end{theorem}
 Combining this result with those of Beauville and Donagi \cite{BD},
 we  conclude that  $VSP(F_{IR}(S),10)$ is isomorphic to the Fano variety of lines in the Pfaffian cubic fourfold $F_{BD}(S)$.
 Theorem \ref{3.7} also says that $VSP(F_{IR}(S),10)$
 is  a smooth hyperk\"ahler fourfold.
 A deformation argument (\cite[proof of Theorem 3.17]{IR}), may therefore be applied
 to prove
 \begin{corollary}{\label{3.8}} For a general cubic fourfold $F$, the variety $VSP(F,10)$ is a smooth and irreducible hyperk\"ahler fourfold.
 \end{corollary}
 \begin{remark} {\rm Note that the statement of \cite[Theorem 3.17]{IR} is incorrect, and was corrected in \cite{IR2}.}
 \end{remark}

Recall from \cite{BD} that the Hodge structure on $H^4(F,\QQ)$, for
$F$ a smooth cubic fourfold, is up to a shift isomorphic to the
Hodge structure on $H^2$ of its variety of lines, the isomorphism
being induced by the incidence correspondence. The construction of
Iliev and Ranestad provides for general $F$ a second hyperk\"ahler
fourfold $VSP(F,10)$ associated to $F$. A natural question is
whether there is also an isomorphism of Hodge structures of bidegree
$(-1,-1)$ between $H^4(F,\QQ)$ and $ H^2(VSP(F,10),\QQ)$. Note that
Theorem \ref{3.7} above combined with the results of Beauville and
Donagi does not imply this statement even for the particular cubic
fourfolds of the type $F_{IR}(S)$, because the Hodge structures on
degree $4$ cohomology of  the cubics $F_{IR}(S)$ and $F_{BD}(S)$
could be unrelated. Another way of stating our question is whether
the two hyperk\"ahler fourfolds associated to $F$, namely its
variety of lines and $VSP(F,10)$, are ``isogenous'' in the Hodge
theoretic sense.

We prove in this paper that such a Hodge correspondence does not
exist for  general $F$.

\begin{theorem}\label{main} For a very general cubic
fourfold $F$, there is no  nontrivial morphism of Hodge structures
$$\alpha:H^4(F,\QQ)_{prim}\rightarrow H^2(VSP(F,10),\QQ).$$

In particular, there is no correspondence $\Gamma\in
CH^3(F\times VSP(F,10))$, such that
$[\Gamma]_*:H^4(F,\QQ)_{prim}\rightarrow
H^2(VSP(F,10),\QQ)$ is non zero.

\end{theorem}
This theorem cannot be proved locally (in the usual topology),
because the  two variations of Hodge structures have the same shape
and  we have no description of the periods of $VSP(F,10)$: it is
even not clear how its holomorphic $2$-form is constructed. In fact,
by the general theory of the period map, there exists locally near a
general point of the moduli space of cubic fourfolds and up to a
local change of holomorphic coordinates, an isomorphism between the
complex variations of Hodge structure on $H^4(F,\CC)_{prim}$ and
$H^2(VSP(F,10),\CC)_{prim}$. Indeed, by the work of Beauville and
Donagi, we know that the variation of Hodge structure on
$H^4(F,\CC)_{prim}$ is isomorphic (with a shift of degree) to the
variation of Hodge structure on $H^2_{prim}$ of the corresponding
family of  varieties of lines, hence in particular this is (up to a
shift of degree) a complete variation of polarized Hodge structures
of weight $2$ with Hodge numbers $h^{2,0}=1$, $h^{1,1}_{prim}=20$.
The same is true for the variation of Hodge structure on
$H^2(VSP(F,10),\CC)_{prim}$ once one knows that the family of
$VSP$'s is locally universal at the general point, which is
equivalent to saying that the deformations of $VSP(F,10)$ induced by
the deformations of $F$ have $20$ parameters, this last fact being
easy to prove. Hence both complex variations of Hodge structures are
given (locally near a general point in the usual topology) by an
open holomorphic embedding
 into a quadric in $\PP^{21}$, and thus they are locally isomorphic  since a quadric is a homogeneous space.

 Notice that if  we consider plane sextic curves instead of cubic fourfolds, then
 we are faced with an analogous situation, namely we can associate naturally
 to a plane  sextic curve $C$ two $K3$ surfaces, the first one being the double cover of $\PP^2$ ramified along
 $C$, and the other one being the variety of power sums $VSP(C,10)$, which
 has been proved by Mukai \cite{Muk} to be
a smooth $K3$ surface for general $C$ (see also \cite{dolgachev}).

 Theorem \ref{main} will be obtained as a consequence of the
 following construction which relates the Mukai construction for
 plane
 sextic curves to the Iliev-Ranestad construction for cubic fourfolds.
 This
involves the introduction of the closed  algebraic subset of the moduli space of the
cubic $F$
parameterizing cubic fourfolds apolar to a Veronese
surface.  This subset,  which we
will prove to be a divisor $D_{V-ap}$, will now be  introduced in more detail.

Let $W$ be a $3$-dimensional vector space, and $V:=S^2W$, which is a
6-dimensional vector space. There is a natural map
$$s:S^6W\rightarrow S^3V$$
which is dual to the multiplication map
$$m:S^3(S^2W^*)\rightarrow S^6W^*.$$
If $a\in W$, we have
\begin{eqnarray}\label{letrivial}
s(a^6)=(a^2)^3.
\end{eqnarray}

 The map $s$ associates to a plane sextic curve $C$ with equation
 $g\in S^6W$ a four dimensional
cubic $F$ with equation $f=s(g)\in S^3V$. Note that we recover $g$
from $f$ using the multiplication morphism $m':S^3V\rightarrow
S^6W$. Indeed  we have \begin{eqnarray}\label{duals} m'(f)=g,
\end{eqnarray}
 as an immediate consequence of  (\ref{letrivial}).

\begin{lemma} \label{leautretrivial} The cubic polynomials in the image of $s$ are exactly those
which are apolar to the Veronese surface $\Sigma\subset
\PP(S^2W)$.
\end{lemma}
\begin{proof} Indeed, by definition of apolarity, a cubic
hypersurface defined by an equation $f\in S^3V$ is apolar to the
Veronese surface if and only if the hyperplane $H_f\subset S^3V^*$
determined by $f$ contains the ideal $I_\Sigma(3)$. Equivalently,
$\langle f,k \rangle =0$, for $k\in I_{\Sigma}(3)$. But as we have $f=s(g)$,
(\ref{duals}) tells that
$$\langle f,k \rangle =\langle g,m(k) \rangle .$$
By definition of the Veronese embedding, the map $m:
S^3V^*\rightarrow S^6W^*$ is nothing but the restriction map to
$\Sigma$, so that $m(k)=0$ and $\langle f,k\rangle =0$ for $k\in I_\Sigma(3)$.
For the converse, note that the map $s$ is injective and that ${\rm dim}_\CC S^6W= {\rm dim}_\CC S^3V^*-{\rm dim}_\CC I_\Sigma(3)$,
so  if $\langle f,k\rangle =0$ for every $k\in I_\Sigma(3)$, then $f$ is in the image of $s$.
\end{proof}

It follows  that the $K3$ surface  $VSP(C,10)$ embeds naturally in
$VSP(F,10)$ and we will prove in Section \ref{secpropsing}:
\begin{theorem}\label{propsingintro} The variety $VSP(F,10)$ is  singular along
$VSP(C,10)$.
For a general choice of $C$, the variety $VSP(F,10)$ is
smooth away from the $K3$ surface $VSP(C,10)$ and has nondegenerate
quadratic singularities along $VSP(C,10)$.
\end{theorem}

Our strategy for the proof of Theorem \ref{main} is the following.
We will first prove that $D_{V-ap}$ is a divisor, and that the
divisor $D_{V-ap}$ is not a Noether-Lefschetz divisor in the moduli
space $\mathcal{M}$ of cubic fourfolds (Proposition
\ref{propnonNL}), which means that for a general cubic parameterized
by this divisor, there is no nonzero Hodge class in
$H^4(F,\QQ)_{prim}$. Secondly, using Theorem \ref{propsingintro}, we
will prove  that $D_{V-ap}$ is a Noether-Lefschetz divisor for the
family $\mathcal{V}SP(F,10)$ of varieties of power sums
 parameterized by a Zariski open set of $\mathcal{M}$, which has to be
interpreted in the sense that the generic Picard rank of the
extension along $D_{V-ap}$ of the variation of Hodge structure on
the degree $2$ cohomology of $VSP(F,10)$ is at least $2$.

Both proofs involve a careful analysis of the variety of power sums
$VSP(F,10)$ with results that we believe may have independent
interest. Indeed, the set theoretic definition given in
(\ref{eq7june}) of $VSP(F,s)$ as a closure in the Hilbert scheme
does not give a priori any information on its schematic structure.
We obtain  in Section \ref{3.2} the following results in the case of
$VSP(F,10)$ for cubic fourfolds.
 Let $U\subset {\rm Hilb}_{10}(\mathbb{P}^5)$ be the open
set of zero-dimensional
 subschemes imposing independent conditions to
cubics. There is vector bundle $E$ of rank $46$ on $U$, with fiber
$I_Z(3)$ over the point $[Z]\in{\rm Hilb}_{10}(\mathbb{P}^5)$.
 \begin{theorem} (i) (cf. Proposition \ref{allapolar}) For a general choice of $F$ in the complement of
explicit divisors in the moduli space of cubic fourfolds, the
variety of power sums $VSP(F,10)$ is contained in $U$ and is the
zero locus of a  section of the vector bundle $E^*$ on $U$.

(ii)  (cf. Proposition \ref{corlea}) For a general cubic fourfold
$F$, the variety $VSP(F,10)$ does not intersect the singular locus
of
 ${\rm Hilb}_{10}(\mathbb{P}^5)$.

 (iii) (cf. Proposition \ref{propapdivapolar} and Corollary \ref{corleb}) These results remain  true  for a general  cubic fourfold
apolar to a Veronese surface.

\end{theorem}

In order to prove these results, we were led  to introduce new
divisors in the moduli space of cubic fourfolds, that is divisors in
$\PP(S^3V)$ invariant under
the action of $PGl(6)$, along which properties stated above fail.
Many $PGl(6)$-invariant divisors were already known: the discriminant hypersurface
parameterizing singular cubic fourfolds and the infinite sequence of
divisors of smooth cubic fourfolds containing a smooth surface which
is not homologous to a complete intersection, introduced by Brendan
Hassett \cite{Ha}. The latter sequence includes  the
Beauville-Donagi hypersurface parameterizing Pfaffian cubics. These
are all Noether-Lefschetz divisors.
 Concerning the new divisors $D_{rk3}$, $D_{copl}$ and
$D_{V-ap}$  we introduce in this paper (see Section \ref{sec1}), we
prove that $D_{V-ap}$ is not a Noether-Lefschetz divisor, and it is
presumably the case that neither $D_{rk3}$ nor $D_{copl}$ are
 Noether-Lefschetz divisors.
 We do not know whether the Iliev-Ranestad divisor $D_{IR}$  parameterizing the Iliev-Ranestad
cubics is a Noether-Lefschetz divisor.
  As a consequence of Theorem \ref{3.7}, the Picard rank of the variety $VSP(F,10)$ jumps to $2$ along
 this divisor.  Therefore proving that $D_{IR}$
 is not a Noether-Lefschetz divisor could have been another approach
 to Theorem \ref{main}.
 
 \begin{acks} We would like to thank an anonymous referee for numerous suggestions that improved the presentation of our proofs.\end{acks}

\subsection{Notation}\label{notation} We give the numerical information of the minimal free resolution of a graded
$S=\CC[x_0,\ldots,x_r]$-module
$$ 0 \leftarrow M \leftarrow F_0  \leftarrow F_1 \leftarrow \ldots  \leftarrow F_n  \leftarrow  0$$
with $F_i = \bigoplus_{j \in \ZZ} \beta_{ij}S(-j)$ in {\it Macaulay2}  notation \cite{MAC2}, i. e. in the form
\[
\begin{matrix} \beta_{00} & \beta_{11} & \beta_{22} & \ldots & \beta_{n,n} \\
\beta_{01} & \beta_{12} & \beta_{23} & \ldots & \beta_{n,n+1} \\
\vdots & \vdots & \vdots & \ldots & \vdots \\
\beta_{0m} & \beta_{1,m+1} & \beta_{2,m+2} &\ldots &\beta_{n,n+m}. \\
\end{matrix}
\]
The $\beta_{0j}$ counts the number of linearly independent generators of $M$ of degree $j+1$,
while the $\beta_{ij}$, for $i>0$  counts the homogeneous sets of linearly independent syzygies of order $i$.

\section{Some divisors in the moduli space of cubic fourfolds \label{sec1}}

Let $V=\CC^6$. We introduce in this section two $PGl(V)$-invariant divisors $D_{rk3}$ and $D_{copl}$
in the open set $\PP(S^3V)_{reg}$ of the projective space
$\PP(S^3V)$
 parameterizing smooth cubic fourfolds. %, which
%are invariant under the $PGl(V)$-action.
We also recall the
definition of the Iliev-Ranestad divisor $D_{IR}$. These divisors
are crucial in the proof that the set $D_{V-ap}$ considered in the
introduction is also a divisor (Corollary \ref{codim1} in Section
\ref{sec3}).

\vspace{1cm}

 {\bf The divisor $D_{rk3}$.} This is the %codimension $1$ component of  the
 set of  cubic forms
 $[f]\in \PP(S^3V)_{reg}$ such that $f$ has a partial derivative of
 rank $\leq 3$.
 \begin{lemma} \label{lediv118aout} The set of  cubic forms
 $[f]\in \PP(S^3V)_{reg}$ such that $f$ has a partial derivative of
 rank $\leq 3$ is an irreducible divisor in
 $\PP(S^3V)_{reg}$.
 \end{lemma}
\begin{proof} If $[f]\in D_{rk3}$, there exist a point $p\in
\PP(V^*)$ and a plane $\PP(W) \subset \PP(V^*)$ such that
\begin{eqnarray}\label{eqn18aout}\frac{\partial^2f}{\partial p\partial w}=0,\,\forall w\in W.
\end{eqnarray}
Consider the case where $p$ does not belong to $\PP(W)$ and let us
compute how many conditions on $f$ are imposed by (\ref{eqn18aout})
for fixed $p,\,W $. We may choose coordinates $X_i,\,i=0,\ldots,5$,
such that  $W$ is defined by $X_i=0,\,i=3,4,5$ and $p$ is defined by
equations $X_i=0,\,i=0,\ldots,4$. Then $f$ has to satisfy the
conditions
$$\frac{\partial^2f}{\partial X_5\partial X_i}=0,\; {\rm for\; any}\,i\in \{0,\,1,\,2\}.$$
Equivalently, we have
\begin{eqnarray}\label{eqn218aout} \frac{\partial^3f}{\partial X_5\partial
X_i\partial X_j}=0, \; {\rm for\; any}\,i\in \{0,\,1,\,2\}\;{\rm and \; any}\; j\in \{0,...,5\}.%\,i=0,\,1,\,2\,\,\forall j.
\end{eqnarray}
 The number of coefficients of $f$ annihilated by these
conditions is $15$. As the pair $(p,W)$ has  $14$ parameters, we
conclude that the $f$ satisfying these equations for some $(p,\,W)$
fill-out at most a hypersurface.  On the other hand, the map
$$\PP(S^3V)_{reg}\to G(6,S^2V); \; [f]\mapsto \langle \frac{\partial f}{\partial X_0},..., \frac{\partial f}{\partial X_5}\rangle$$
is generically injective; for general $f$, the apolar ideal is generated by the quadrics orthogonal to the partials of $f$, and according to Macaulays theorem, the apolar ideal defines $f$ up to scalar. The rank $3$ locus in $\PP(S^2V)$ has codimension $6$, so the $6$-dimensional subspaces of $S^2V$ that intersect the rank $3$ locus form a hypersurface section in $G(6,S^2V)$.  Therefore the cubic forms that have a partial of rank $3$ form at least a divisor in $\PP(S^3V)_{reg}$, i.e. they form exactly a divisor.
It is irreducible, because
it is dominated by a projective bundle over the parameter space for
$(p,W)$.  Denote this hypersurface by $D_{rk3}$.

To complete the argument we consider the degenerate situation where $p\in \PP(W)$. It may be seen as a limit of the above case:
 We may choose coordinates $X_i,\,i=0,\ldots,5$,
such that  $W$ is defined by $X_i=0,\,i=3,4,5$ and $p_t$ is defined by
equations $X_i=0,\,i=1,\ldots,4$ and $X_5=tX_0$. Thus $p_0\in  \PP(W)$.  For any $t$, we consider the cubic forms $f$ that satisfy the
conditions
$$\frac{\partial^2f}{\partial X_0\partial X_i}-t\frac{\partial^2f}{\partial X_5\partial X_i}=0,\,i=0,\,1,\,2.$$
Equivalently, we have
\begin{eqnarray}\label{eqn318aout} \frac{\partial^3f}{\partial X_0\partial
X_i\partial X_j}-t\frac{\partial^3f}{\partial X_5\partial
X_i\partial X_j}=0,\,i=0,\,1,\,2\,\,\forall j.
\end{eqnarray}
These are $15$ linearly independent conditions on the coefficients of $f$ for any value of $t$. In particular, any cubic form $f_0$ satisfying the conditions with $t=0$ is a limit of forms $f$ that satisfy the conditions for $t\not=0$ as $t$ tends to $0$.
So also in the degenerate situation, the forms lie in the irreducible hypersurface $D_{rk3}$.

\end{proof}

Note the following other characterization of $D_{rk3}$:
\begin{lemma}  \label{lemmatouselater19aout} A cubic form belongs to $D_{rk3}$ if it has
a net (a $3$-dimensional vector space) of partial derivatives which
are all singular in a given point $p$.
\end{lemma}
\begin{proof} The fact that $f$ has
a net of partial derivatives which
are singular in a point $p$ is equivalent to the vanishing
$\partial_p(\partial_{w_i}f)=0$ for three independent vectors $w_i$.
This holds if and only if  $\partial_{w_i}(\partial_pf)=0$ for
$i=1,2,3$, which in turn is equivalent to the fact that the partial
derivative $\partial_pf$ has rank $\leq3$.
\end{proof}

{\bf The divisor $D_{copl}$.} The subset $D_{copl}\subset\PP(S^3V)_{reg}$
is the Zariski closure of the set of forms $f$ which can be written as
\begin{eqnarray}\label{eqn4218aout}f=\sum_{1=1}^{10}a_i^3,
\end{eqnarray} such that four of the linear forms $a_i\in V$
are coplanar.

\begin{lemma} \label{lediv218aout} $D_{copl}$ is an irreducible divisor in
 $\PP(S^3V)_{reg}$.
 \end{lemma}
\begin{proof} The set $D_{copl}$ is irreducible, since it is dominated
by the irreducible algebraic set parameterizing the $10$ linear
forms, four of which are coplanar. If we count dimensions, we find
that this last algebraic set has dimension $56$. However, we observe
that a general cubic form $g$ in $3$ variables has a two dimensional variety
of power sums $VSP(E,4)$, where $E=V(g)$.  If $f=\sum_{1=1}^{i=10}a_i^3$, where
$a_1,\ldots,a_4$ are coplanar, we  have
\begin{eqnarray}\label{eqn5218aout} f=g(b_1,b_2,b_3)+\sum_{i=5}^{i=10}a_i^3,
\end{eqnarray}
 where the $a_i$'s
for $i\leq 4$ are linear combinations of the $b_i$'s. As there is a
2-parameter family of ways of writing $g$ as a sum of four powers of
linear forms in the $b_i$'s, we conclude that there is a 2-parameter
family of ways of writing $f$ as in (\ref{eqn4218aout}). This proves
that $D_{copl}$ has codimension at least $1$.
To show that it actually is a divisor, we exhibit an affine subfamily of $D_{copl}$ of codimension one in the space of cubic forms.
In fact if we let $$b_1=x_0+b'_0,b_2=x_1+b'_2,b_3=x_2+b'_3$$
and $$a_5=x_0-x_1+x_3+x_4+a'_5, a_6=x_1+x_2-x_3-x_4-x_5+a'_6,$$
$$a_7=x_2+x_3-x_4+x_5+a'_7, a_8=x_3+a'_8, a_9=x_4+a'_9,
a_{10}=x_5+a'_{10},$$
with $b'_1,..,b'_3,a'_5,...,a'_{10}\subset V$,
then
$$f=g(b_1,b_2,b_3)+\sum_{i=5}^{i=10}a_i^3$$
 belongs to $D_{copl}$ for every 9-tuple of linear forms $b'_1,..,b'_3,a'_5,...,a'_{10}$.
 The summands in $f$ that are linear in the $b'_i$ and $a'_j$ span the tangent space to this  family at the origin, where $b'_1=...=a'_{10}=0$.  This space may thus be shown, with {\it Macaulay2} \cite{MAC2}, to have dimension $54$.    Therefore the family is a divisor.
\end{proof}

{\bf The divisor $D_{IR}$.} This is the divisor constructed by Iliev
and Ranestad in \cite{IR}. It parameterizes the cubic fourfolds
$F_{IR}(S)$ mentioned in the introduction, associated to $K3$
surfaces $S$ which are complete intersections of the Grassmannian
$G(2,6)\subset \PP^{14}$ with a $\PP^8_S$. More precisely, these
cubic fourfolds are defined as follows: Dual to  $\PP^8_S$, we get a
$\PP^5_S\subset \check{\PP^{14}}$. The dual projective space
$\check{\PP^{14}}$ contains
 the Grassmannian of lines $\check{G}(2,6)$ and for generic choice of $\PP^5_S$, the intersection
$\PP^5_S\cap \check{G}(2,6)$ is empty. It is then proved in \cite{IR} that
the ideal of cubic forms on  $\check{\PP^{14}}$ vanishing on $\check{G}(2,6)$
restricts to a hyperplane in
$H^0(\PP^5_S,\mathcal{O}_{\PP^5_S}(3))$. This hyperplane in turn
determines a cubic fourfold in $\check{\PP^5_S}$.

 For later use in the paper,  we
recall and extend a characterization from \cite{IR} of apolar length
$10$ subschemes to cubic forms $[f]\in D_{IR}$ in terms of quartic surface scrolls, i.e. rational normal surface scrolls in $\PP^5$.
\begin{lemma}\label{DIR}  Let $f$ be a cubic form of rank $10$, such that $[f]\in D_{IR}$.
Then the general  subscheme of length $10$ apolar to $f$ is the
intersection of two quartic surface scrolls. In particular $f$ is apolar to
a quartic surface scroll.

Conversely, if $f$ is a cubic form of rank $10$ apolar to a quartic
surface scroll, then $[f]\in D_{IR}$.
\end{lemma}
\begin{proof} The first part is shown in \cite {IR}:
Let $S=\GG(2,6)\cap\PP^8_S$ be the $K3$-surface section associated to $F=V(f)$,
 i.e. $F=F_{IR}(S)$ in the notation of loc. cit.  Then $S$ parameterizes quartic surface scrolls apolar to $f$, and
the two scrolls corresponding to a pair of points on $S$ intersect in a length $10$ subscheme apolar to $f$
(Lemma 2.9 and the proof of Theorem 3.7 loc.cit.).

For the second part, if $f$ is apolar to a quartic surface scroll, then by
dimension count, $f$ has a $2$-dimensional family of length $10$
apolar subschemes on this scroll.  The general such subscheme $Z$  has a
Gale transform in $\PP^3$ contained in a smooth
 quadric surface \cite[Corollary 3.3]{EP}.  Furthermore, the two rulings
 in the quadric surface correspond to two quartic surface scrolls that contain $Z$,   see \cite[Example 3.4]{EP}, where an analogous case is explained.  Therefore $f$ is apolar to a $2$-dimensional family of quartic surface scrolls.
 Now, the family of quartic surface scrolls in $\PP^5$ is irreducible of dimension $29$, and each scroll
 is apolar to a $27$-dimensional space of cubic forms, so there is an irreducible $54$-dimensional family of cubic forms
 apolar to some quartic surface scroll.  This family must  coincide with the divisor $D_{IR}$ since it contains it.
\end{proof}
 \section{Apolarity and syzygies \label{3.2}}

In this section we first show that for a general cubic fourfold $F\subset \PP(V^*)$, the
variety $VSP(F,10)$ is defined as the zero locus, inside the Hilbert
scheme, of a section of a vector bundle.  In fact the variety
$VSP(F,10)$ is then entirely contained in the set $U\subset {\rm
Hilb}_{10}(\PP(V))$ of zero-dimensional subschemes imposing
independent conditions on cubics (Proposition \ref{allapolar}), and $Z$ is apolar to $F$ for every $[Z]\in VSP(F,10)$. Furthermore,
after defining the cactus rank of a cubic fourfold $F$  (Definition \ref{defcactusrank}), we note that any scheme of minimal length apolar to $F$, is locally Gorenstein, and show, as a consequence, that $VSP(F,10)$ does not meet the singular locus of ${\rm Hilb}_{10}(\PP(V))$ for a general $F$ (Proposition \ref{corlea}).
We also show that if $F$ is general,
then the cactus rank coincides with the rank and $VSP(F,10)$
contains all schemes of length $10$ that are apolar to $F$ (Corollary \ref{connected}).

In the second part of this section we give a criterion (Lemma \ref{rank9}) for a cubic form $f$ to have
cactus rank $10$
 in terms of a syzygy variety of its apolar ideal $I_f$.
 When a cubic fourfold $F\subset \PP(V^*)$ has cactus rank $10$,
  then the union of the apolar subschemes of length~$10$ forms a hypersurface $V_{10}(F)$
   in $\PP(V)$.  We will show (Lemma \ref{leautiliser}) that
  $V_{10}(F)$ is a syzygy variety of $I_f$, and analyze its singular locus.
At the end of this section we show (Proposition \ref{corlea}) that $VSP(F,10)$ does not meet the singular locus of ${\rm Hilb}_{10}(\PP(V))$ for a general $F$.

The results of this section that are used later, are formulated in two lemmas and two propositions.
Lemmas \ref{rank9} and \ref{leautiliser} will be used in Section \ref{sec3} to prove that a general
$[f]\in D_{V-ap}$ is apolar to finitely many Veronese surfaces, from
which we will deduce that $D_{V-ap}$ is a divisor.  Propositions \ref{allapolar} and \ref{corlea} are applied in Section \ref{sec3} to show that for a general $[f]\in D_{V-ap}$,
the length $10$ subscheme $Z$ is apolar to $f$ for every $[Z]\in  VSP(F,10)$ and is a smooth point in  ${\rm Hilb}_{10}(\PP(V))$.

\subsection{Apolar subschemes of length 10}

\begin{proposition}\label{allapolar} Let $F\subset \PP(V^*)$ be a cubic
 fourfold defined by a general form $f \in {\rm Sym}^3V$.
  Then any length
$10$ subscheme $[Z]\in VSP(F,10)$ imposes independent conditions to
cubics, i.e. $h^1(\mathcal{I}_Z(3))=0,$
%In particular, for any $[Z]\in VSP(F,10)$, the scheme $Z$
and is apolar to $f$, that
is $I_Z(3)\subset H_f$.

Furthermore,  if there is a codimension $1$ component of the set of
smooth cubic fourfolds not satisfying this conclusion, it must be
one of the two divisors $D_{rk3}$ and $ D_{copl}$  introduced in the
previous section.
\end{proposition}
Note that the second statement follows from the first using Lemma \ref{3.4}
and the fact that the condition $I_Z(3)\subset H_f$ is a closed
condition on the open set $U\subset {\rm Hilb}_{10}(\PP(V))$
of zero-dimensional subschemes imposing independent conditions to
cubics.

The proof of Proposition \ref{allapolar} is postponed until later in this section.
 The proposition %\ref{allapolar}
 will be crucial in the study of the schematic
structure of $VSP(F,10)$, for $f$ satisfying the above conditions.
To see this, we first consider finite subschemes of minimal length apolar to $f$.
%Indeed, it %immediately
%implies:
A form $f$ of rank $10$ may be apolar to subschemes of length less than $10$.  This motivates the
notion of  {\it cactus rank} of $f$:
\begin{definition} \label{defcactusrank} The cactus rank of a form $f$ or equivalently of the hypersurface $F=V(f)\subset{\bf
P}^n$ is the minimal length of a $0$-dimensional subscheme $Z$ of
$\check{{\bf P}^n}$ which is apolar to $f$ (resp. $F$).
\end{definition}

\begin{remark}\label{BC}\hfill
{\rm
\begin{enumerate}
\item Buczy\'{n}ska and Buczy\'{n}ski showed in
\cite[Proposition 2.2, Lemma 2.3]{BB} that a finite subscheme $Z$, that is apolar to $f$ and has length equal to its cactus rank, is locally Gorenstein.
% by
 %\cite[Proposition 2.2, Lemma 2.3]{BB}.%,  a scheme $Z$ that is apolar to $F$ and has length equal to the cactus rank of $F$ is locally Gorenstein.
 \item    Casnati, Jelisiejew and Notari have shown that any local Gorenstein scheme of length
   at most $13$ is smoothable (cf. \cite[Theorem A]{CJN}).
      \end{enumerate} }
   \end{remark}
  Since the smooth apolar schemes form an open set in its component of the Hilbert scheme, we get:
  \begin{lemma} \label{cactusvsrank}
 If $F$ is a general cubic fourfold of rank $10$, then the cactus rank of $F$ is also $10$. %and $VSP(F,10)$ %is irreducible and
% contains all subschemes of length $10$ that are apolar to $F$.  %By Proposition \ref{allapolar}
\end{lemma}
\begin{proof} Since Gorenstein schemes of length $\leq 9$ are smoothable, cubic forms $f$ of cactus rank $\leq 9$ lie in the closure of forms of rank $\leq 9$.  But the closure of the set of forms of rank $\leq 9$ is a proper subset of the set of cubic forms,  so the general form of rank $10$ must also have cactus rank $10$.
\end{proof}

The Proposition \ref{allapolar} provides a criterion for $VSP(F,10)$ to avoid the singular locus of
${\rm Hilb}_{10}\PP(V)$.%, that we will also need in the
% next section.

\begin{proposition}\label{corlea}
Let $V=\CC^6$, and let $F$ be a fourfold defined by a cubic form
$f\in {\rm Sym}^3V$  with no partial derivative of rank $\leq3$.
If $f$ has cactus rank $10$ and $Z$ is apolar to $f$ for every $[Z]\in  VSP(F,10)$,
then $VSP(F,10)$ does not intersect  the singular
locus of ${\rm Hilb}_{10}(\PP(V))$.
\end{proposition}
\begin{proof} Let  $[Z]\in  VSP(F,10)$, then, by  Remark \ref{BC}, the scheme $Z$ is locally Gorenstein.
Consider the morphism $q_f:\PP(V)\to \PP(Q_f^*)$ defined by the space of quadrics $Q_f$
that are apolar to $F$.  Then the linear span of the image $q_f(Z)$ has, by  Lemma \ref{lenew25juillet}, dimension $2$ or $3$.
Since $f$ has no partial of rank $\leq3$, the morphism $q_f$ is, by Lemma \ref{q_f}, an embedding, so the scheme $Z$ is embeddable in $\PP^3$.
 By \cite{HK} and  \cite[Corollary 2.6]{CN}, the corresponding point $[Z]$ is smooth in the Hilbert scheme.
\end{proof}

By Remark \ref{BC}, the open set $U_G\subset U\subset {\rm Hilb}_{10}(\PP^5)$ of length $10$ locally Gorenstein subschemes that impose independent conditions to cubics is contained in the irreducible component of the smooth subschemes.
  
\begin{corollary} \label{connected} Let $F=V(f)$ be a general cubic fourfold. Then $VSP(F,10)$ is   %and connected
%component of
the zero locus
 of a section $\sigma_f$ of the vector bundle $\mathcal{E}$
on $U_G$ of rank $46$ with fiber $I_Z(3)^*$.  In particular, $VSP(F,10)$ admits a natural smooth and connected scheme structure and contains all subschemes of length $10$ that are apolar to $F$.
\end{corollary}
\begin{proof} Indeed, let  $\sigma_f$ be the section of $\mathcal{E}$
given by $Z\mapsto f^*_{\mid I_Z(3)}$, where $f^*$ denotes the
linear form on ${\rm Sym}^3V^*$ corresponding to $f$. Then
$\sigma_f$  vanishes on $VSP(F,10)$ by Proposition
\ref{allapolar}.  The set $U_G$ is irreducible and the set of sections $\sigma_f$ clearly has no basepoints.  By Proposition \ref{corlea}, the general section vanishes only in the smooth locus of $U_G$, so the zero locus of $\sigma_f$ is smooth and connected for general $F$.  
\end{proof}

The proof of Proposition \ref{allapolar}
will need a few preparatory lemmas.

For a cubic form $f\in S^3V$ such that $F=V(f)$ is not a cone,  let  $P(f)\subset \PP(S^2V)$ be the
space of partial derivatives of $f$ and $Q_f=P(f)^{\bot}\subset
S^2V^*$.  Then $P(f)$ is $6$-dimensional and hence ${\rm dim}\,Q_f=15$.
Note that $Q_f=[H_f:V^*]$,
where $H_f\subset S^3V^*$ is the hyperplane defined by $f^*$; indeed we may
identify the space of partials $P(f)$ with the image $V^*(f)\subset S^2V$, so  if $q\in S^2V^*$, then $q\cdot V^*(f)=0$ if and only if
$q(P(f))=0$.

Consider now a subscheme $Z\subset \mathbb{P}^5$ of length $10$.  Since $Z$ imposes at most $10$ conditions on quadrics, the space $I_Z(2)$ of quadrics in the ideal has dimension at least $11$, with equality for an open set of schemes $Z$.   Likewise, the ideal is generated in degree $2$, for an open set of length $10$ schemes $Z$: If $Z$ is the intersection of a rational normal quintic curve and a quadric, then $I_Z(2)$ has dimension $11$ and generate the ideal $I_Z$.  Therefore this is the case also for a general $Z$.

Thus, in particular, if $F$ is a general cubic fourfold and $[Z]\in VSP(F,10)$ is general, then $I_Z(2)$ has dimension $11$ and generate the ideal $I_Z$.
%The ideal $I_Z(2)$ of a general  $[Z]\in VSP(F,10)$ is generated by $11$
%quadrics in $S^2V^*$ and
 Furthermore, by Lemma \ref{3.4},  $I_Z(2)\subset Q_f$.
It follows that the rank of the evaluation map $$Q_f\rightarrow
H^0(\mathcal{O}_Z(2))$$ is at most $4$ for a general $[Z]\in
VSP(F,10)$, and by semicontinuity of the rank, the same remains true
for any $[Z]\in VSP(F,10)$.  Therefore
 \begin{lemma}\label{Z_f} Let $f\in S^3V$  be a cubic form such that $F=V(f)$ is not a cone, and let $[Z]\in VSP(F,10)$, then
${\rm dim}\,I_Z(2)\cap Q_f\geq 11$.
  \end{lemma}

 The linear system of quadrics $Q_f$ gives a rational map
\[
q_f:\PP(V)\dashrightarrow \PP(Q_f^*),
\]
defined as
 the composition of the Veronese map $\PP(V)\to\PP(S^2V)$ and the projection from the subspace  $P(f)\subset \PP(S^2V)$.

The following lemma is an immediate consequence of this description.
 \begin{lemma}\label{q_f}\hfill
 \begin{enumerate}
\item\label{x} $q_f$ is a morphism if and only if $f$ has no partials of rank $\leq 1$.%(since otherwise one partial derivative of $f$ would be of rank $1$)

\item\label{xx}$q_f$ is an embedding if and only if $f$ has no partials of rank $\leq 2$.

\item\label{xxx} $q_f$ is an embedding and
  the image $X_f:=q_f(\PP(V))$ contains no subscheme of length $3$ contained in a line
if and only if
  $f$ has no partial derivative of rank  $\leq 3$, i.e. $f\notin D_{rk3}$.
  \end{enumerate}
  \end{lemma}

  This lemma allows us to find possible schemes $Z$ such that ${\rm dim}\,I_Z(2)\cap Q_f\geq 11$.
  \begin{lemma}\label{lenew25juillet} Let $f$ be a cubic form with no partial derivative
  of rank $\leq 3$, let $X_f:=q_f(\PP(V))$ be the image by the map $q_f$ and let
 $P\subset \PP(Q_f^*)$ be a $\PP^3$.
 If $X_P:=P\cap X_f$ contains a curve, then  $X_P$ is the image by $q_f$ of a line and a residual finite subscheme.

   In particular, if $F=V(f)$, $[Z]\in VSP(F,10)$ and $Z_f=q_f(Z)$, then the linear span of $Z_f$ is a $\PP^2$ or a $\PP^3$,
   and if $I_Z(2)\cap Q_f$ is contained in the ideal of a curve, this curve is a line.    \end{lemma}
  \begin{proof}  Indeed, by Lemma \ref{q_f} \eqref{xxx}, $q_f$ is an embedding and the image $X_f$ has no trisecant line.  Since it is a linear projection of the second Veronese embedding, every curve in the image has even degree.    Consider now a $3$-space $P\subset \PP(Q_f^*)$  and the intersection $X_P=P\cap X_f$.
  Since every surface in $P$ contains a line or has a trisecant line,  $X_P$ cannot contain a surface. Furthermore, the only curves in $P$ of even degree with no trisecant lines are the conics and the complete intersections of two quadric surfaces (e.g. \cite{ball}). 
  But a complete intersection of two quadric surfaces is not the second Veronese
  embedding of a curve.  Therefore, if $X_P$ contains a curve, $X_P$ is the union of a conic and a residual finite subscheme.

  If  $[Z]\in VSP(F,10)$, then ${\rm dim}\,I_Z(2)\cap Q_f\geq 11$ by Lemma \ref{Z_f}, so the span $\langle Z_f \rangle $ is at most a $\PP^3$. On the other hand, $Z_f$ must span at least a
  plane, since $X_f$ has no trisecant line, so that $3\geq{\rm dim}\,\langle Z_f \rangle \geq 2$.  The linear span  $\langle Z_f \rangle $ intersects $X_f$ in the zero locus of $I_Z(2)\cap Q_f$, so the last claim in the lemma now follows from the first.
  \end{proof}

Notice that the span $\langle Z_f \rangle $, whether $Z$ is apolar to $f$ or not,  has dimension $2$ (resp. $3$) if and
only if $I_Z(2)\cap Q_f$ has dimension $12$ (resp. $11$).

\begin{lemma}\label{leb2}   Let $V=\CC^6$, and let
 $f\in {\rm Sym}^3V$ be a cubic form  with no partial derivative of rank $\leq3$.
 Let $Z\subset \PP(V)$ be a subscheme of length $10$, and assume that
 $I_Z(3)$ has codimension at most $9$ in ${\rm Sym}^3V^*$.  Let $\Gamma\subset \PP(V)$ be the zero locus of the space of quadrics  %the base locus of the space of quadrics
$I_Z(2)\cap Q_f$.

\begin{enumerate}
\item  If ${\rm dim}\,I_Z(2)\cap Q_f=12$, then $\Gamma$ is a line. % If $P$ is a plane, then $Z$ is contained in a line.
 \item  If ${\rm dim}\,I_Z(2)\cap Q_f=11$,  then  $\Gamma$ is the union of a line and a residual finite subscheme.
    \end{enumerate}
   \end{lemma}
\begin{proof}
Let $Z\subset \PP(V)$ be a subscheme of length $10$ and assume that $I_Z(3)$ has codimension at most $9$ in ${\rm Sym}^3V^*$.
Notice first that ${\rm dim}\,I_Z(2)\geq 12$.
In fact, the subscheme $Z$ does not impose independent conditions on cubics, i.e.
$h^1(\mathcal{I}_Z(3))>0.$  The multiplication by a general linear form $h$ defines an exact  sequence of sheaves
\[
 0\to \mathcal{I}_{Z}(2)\to \mathcal{I}_Z(3)\to \mathcal{O}_{H}(3)\to 0,
\]
where $H=\{h=0\}$.   Since  $h^1(\mathcal{O}_H(3))=0$, $h^1(\mathcal{I}_Z(3))>0$ implies that
$h^1(\mathcal{I}_Z(2))>0$, and hence that ${\rm dim}\,I_Z(2)\geq 12$.

Now, assume furthermore that ${\rm dim}\,I_Z(2)\cap Q_f\geq 11$.
Let $\Gamma\subset \PP(V)$ be the zero locus of the space of quadrics  %the base locus of the space of quadrics
$I_Z(2)\cap Q_f$.  Then, $q_f(\Gamma)$ is contained in a $\PP^3$, so by Lemma \ref{lenew25juillet}, $\Gamma$ is either a line and a
residual finite subscheme, or $\Gamma$ is finite.
%Since $F$ is not a cone, $G$ must span $\PP(V)$.

Assume first that $\Gamma$ is finite.  Then $Z$ spans at least a $\PP^4$ in $\PP(V)$,  since any finite intersection of quadrics in a $\PP^3$ has length at most $8$.
Let $Z_0$ be a maximal length subscheme of $Z$ that spans a $\PP^3$
in $\PP(V)$.   The length of $Z_0$ is then at most $8$, and at least $4$ since it spans $\PP^3$.

The residual scheme $Z_1=Z\setminus Z_0$ therefore has length at least $2$ and at most $6$.   %, and the quadrics in the ideal of $Z_1$ has no .
Let $H=\{h=0\}$ be a general hyperplane that contains $Z_0$.
Then  multiplication by $h$ defines a  sequence of sheaves of ideals
\[
 0\to \mathcal{I}_{Z_1}(2)\to \mathcal{I}_Z(3)\to \mathcal{I}_{H,Z_0}(3)\to 0,
\]
which is exact. Since $h^1(\mathcal{I}_Z(3))>0$, either
$h^1(\mathcal{I}_{H,Z_0}(3))>0$  or  $h^1(\mathcal{I}_{Z_1}(2))>0$.

We claim  $h^1(\mathcal{I}_{Z_1}(2))=0$.
  Since $\Gamma$ is a finite intersection of quadrics, the subscheme $Z_1$ contains no subscheme of length $3$
contained in a line, and no subscheme of length $5$ contained in a plane.
By the maximality of $Z_0$, it has at most a subscheme of length $5$ in  a $\PP^3$.

Therefore $Z_1$ either has minimal length in its span, in which case the claim follows, or it has length $d$ in a $\PP^{d-2}$ with $d=4,5$ or $6$.  If $Z_1$ has length $4$ in a plane it is a complete intersection of two curves of degree $2$, so again $h^1(\mathcal{I}_{Z_1}(2))=0$. 
If $Z_1$ has length $5$ and spans a $\PP^3$  or length $6$ and spans a $\PP^4$, it contains a subscheme $Z_2$ of length $3$ or $4$ in a plane $P_2$.  The residual scheme $Z_{1,2}$  to $Z_2$ in $Z_1$ has length $1, 2$ or $3$.
 Multiplication by a general linear form $h$ that contains the  plane $P_2$  defines an exact sequence of sheaves
\[
 0\to \mathcal{I}_{Z_{1,2}}(1)\to \mathcal{I}_{Z_1}(2)\to \mathcal{I}_{H,Z_{2}}(2)\to 0.
\]
Now, $h^1(\mathcal{I}_{Z_{1,2}}(1))=h^1(\mathcal{I}_{H,Z_{2}}(2))=0$, so we infer $h^1(\mathcal{I}_{Z_1}(2))=0$.

We may therefore assume
$h^1(\mathcal{I}_{H,Z_0}(3))>0$.  If $P=\langle Z_0\rangle$, then, by further restriction, also $h^1(\mathcal{I}_{P,Z_0}(3))>0$.  If $Z_0$ has length $4$ or $5$, we may argue as for $Z_1$ above that $h^1(\mathcal{I}_{Z_0}(2))=0$  and hence also $h^1(\mathcal{I}_{Z_0}(3))=0$. So we may assume that $Z_0$ has length at least $6$.
Since $Z_0$ is contained in a finite intersection of quadrics, a general net of these quadrics defines a complete intersection $Y$ in $P$ that contains $Z_0$.  Then $Y$ has length $8$, and contains a subscheme of length at most $2$ residual scheme to $Z_0$.  If $Z_0=Y$, then $h^1(\mathcal{I}_{H,Z_0}(3))=0$, a contradiction.  If $Z_0$ has length $7$ it is residual to a point $p$ in $Y$. Let $X$ be a cubic surface that contains $Z_0$ but not $Y$.  Then multiplication by the form defining $X$ defines two exact sequences
\[
 0\to \mathcal{I}_{p}\to \mathcal{I}_{Y}(3)\to \mathcal{I}_{X,Z_0}(3)\to 0
\]
and
\[
 0\to \mathcal{O}_{P}\to \mathcal{I}_{P, Z_0}(3)\to \mathcal{I}_{X,Z_0}(3)\to 0.
\]
From the first we deduce that $h^1(\mathcal{I}_{X,Z_{0}}(3))=0$, and so by the second  $h^1(\mathcal{I}_{P, Z_{0}}(3))=0$, a contradiction.

If $Z_0$ has degree $6$,
it contains a subscheme $Z_2$ of length $3$ or $4$ in a plane $P_2$.  The residual scheme $Z_{0,2}$  to $Z_2$ in $Z_0$ has length $2$ or $3$.
 Multiplication by the linear form $h$ that defines the  plane $P_2$  defines an exact sequence of sheaves
\[
 0\to \mathcal{I}_{P,Z_{0,2}}(2)\to \mathcal{I}_{P,Z_0}(3)\to \mathcal{I}_{P_2,Z_{2}}(3)\to 0.
\]
Now, $h^1(\mathcal{I}_{P,Z_{0,2}}(2))=h^1(\mathcal{I}_{P_2,Z_{2}}(3))=0$, so we infer $h^1(\mathcal{I}_{P,Z_0}(3))=0$, a contradiction.

 Therefore $\Gamma$ contains a line $\Delta$.  Let $Z_\Delta=Z\cap \Delta$.
 The line $\Delta$ is mapped to a conic $q_f(\Delta)$.  If ${\rm dim}\,I_Z(2)\cap Q_f=12$, then $Z_f=q_f(Z)$ spans only a plane, and the image $q_f(\Gamma)$ has a subscheme of length $3$ in a line, unless $Z$ is entirely contained in $\Delta$, i.e. $Z_\Delta=Z$ and $\Gamma=\Delta$.
\end{proof}

\begin{proof}[Proof of Proposition \ref{allapolar}] Let $[Z]\in
VSP(F,10)$. We assume, for contradiction, that $Z$ does not impose
independent conditions on cubics. Assuming $f$ is regular  and has no
partial derivative of rank $\leq3$, we already proved that
 $12\geq {\rm dim}\,I_Z(2)\cap Q_f\geq 11$.
By Lemma \ref{leb2}, we conclude in both cases that there is a line
$\Delta$ such that $I_Z(2)\subset I_\Delta(2)$, so that $$I_Z(2)\cap
Q_f \subset I_\Delta(2)\cap Q_f.$$ Note also that, under the same
assumptions on $f$, the image $q_f(\Delta)$ is a conic curve in a plane that does not have any residual intersection with $X_f=q_f(\PP(V))$.
Thus ${\rm dim}\,I_\Delta(2)\cap Q_f=12$ and the zero locus of $Q_{f,\Delta}: =I_\Delta(2)\cap Q_f$ is $\Delta$.

 Since $[Z]\in
VSP(F,10)$, there exists a flat family of subschemes
$$(Z_t)_{t\in B},\,Z_t\subset \mathbb{P}^5, \,\text{length}\,Z_t=10,$$
where $B$ is a smooth curve,  such that $Z_0=Z$ for some point $0\in
B$ and for general $t\in B$, $Z_t$ is apolar to $f$ and imposes $10$
independent conditions to quadrics. The subspace
$J_t:=I_{Z_t}(2)\subset Q_f$ is thus of codimension $4$. Let
$J\subset Q_f\cap I_Z(2)$ be the specialization of $J_t$ at $t=0$.
Then ${\rm dim}\,J=11$ and
 $J\subset Q_{f,\Delta}=I_\Delta(2)\cap Q_f$ so that
 $J$ is a hyperplane in $Q_{f,\Delta}$.

 On the other hand, note that
 by semicontinuity of the rank,
 we have for any $k\geq 0$
 $${\rm codim}\,(S^kV^*\cdot J\subset S^{k+2}V^*)\geq {\rm codim}\,(S^kV^*\cdot J_t\subset S^{k+2}V^*)$$
 $$\geq {\rm
 codim}\,(I_{Z_t}(k+2)\subset S^{k+2}V^*)=10.$$

 The contradiction that concludes the proof of Proposition  \ref{allapolar} is derived from the following statement:
 \begin{lemma}\label{propHgen} Assume $f$ is general. Then for any line
 $\Delta \subset \mathbb{P}^5$,
 and for any hyperplane $J\subset Q_{f,\Delta}:=I_\Delta(2)\cap Q_f$,
we have
$${\rm codim}\,(S^3V^*\cdot J \subset S^5V^*)\leq 9.$$
Furthermore, the locus of smooth cubic fourfolds not satisfying this
condition has codimension $>1$ away from the union of $D_{rk3}$ and
$D_{copl}$.
 \end{lemma}
\end{proof}
\begin{proof}[Proof of Lemma \ref{propHgen}]
The proof has two parts, that both depend on the following  property of the zero locus $\Gamma\supseteq \Delta$ of $J$.

Let $\tau_0:X_0\rightarrow \PP^5$ be the blow-up of $\PP^5$ along
$\Delta$. Then $J$ provides a space $J'$ of sections  of
$L_0:=\tau_0^*(\mathcal{O}_{\mathbb{P}^5}(2))(-E_\Delta)$ on $X_0$,
where $E_\Delta$ is the exceptional divisor of $\tau_0$.
\begin{sublemma} \label{isoresiduels} Assume $f$ is regular and has no partial
 derivative of rank $\leq3$. Let $J\subset I_\Delta(2)\cap Q_f$ be a hyperplane with zero locus $\Gamma\supseteq \Delta$.
 Let $H\subset\mathbb{P}^5$ be a hyperplane that does not contain $\Delta$.  Then the subscheme of $H\cap \Gamma$ that has support on $\Delta$ has length at most $2$.
  \end{sublemma}
 \begin{proof}
Since
$f$ has no partial derivative of rank $\leq 3$, the line $\Delta$ is the zero locus of $Q_{f,\Delta}$, i.e. $Q_{f,\Delta}$
generates $\mathcal{I}_\Delta(2)$ at any point of $\Delta$.   Let $E_{\Delta,x}$ be the fiber over $x=H\cap\Delta$ in $E_\Delta$.  Then $E_{\Delta,x}\cong\PP^3$  and
$L_0|_{E_{\Delta,x}}\cong{\mathcal O}_{\PP^3}(1)$. The restriction of the sections $J'$ generates at least a hyperplane of sections in this line bundle,
 so their zero locus on $E_{\Delta,x}$ is at most a point.  So $J$ restricted to $H$, defines a scheme at $x$ that is the intersection of quadrics and is contained in a line, so it has length at most $2$.
 \end{proof}
Now, we first deal with the case
where the zero locus of $J\subset I_\Delta(2)\subset S^2V^*$ has a finite subscheme of length at most $3$ residual to $\Delta$.
 In this case, we have the following:
 \begin{sublemma} \label{lemm3ppooints residuels} Assume $f$ is regular and has no partial
 derivative of rank $\leq3$. Let $J\subset I_\Delta(2)\cap Q_f$ be as above, with zero locus $\Gamma\supseteq \Delta$.
 Assume the scheme
 $\gamma$
 residual to $\Delta$ in $\Gamma$ is finite of length at most $3$. Then
\begin{eqnarray}\label{eq2407}S^3V^*\cdot J=I_\Gamma(5).
\end{eqnarray} In particular, ${\rm codim}\,(S^3V^*\cdot J\subset S^5V^*)\leq9$.
 \end{sublemma}
\begin{proof} Let $\tau_0:X_0\rightarrow \PP^5$ be the blow-up of $\PP^5$ along
$\Delta$, and let, in the notation as above, 
$\gamma'$ be the zero-locus of $J'$ supported over $\gamma$.  As in the proof of Sublemma \ref{isoresiduels}, $\gamma'$ intersects the fiber in $E_\Delta$ over any point  of $\Delta$ in at most a point.  Via the blowup map $\tau_0$, the
subscheme $\gamma'$ is therefore isomorphic to the subscheme $\gamma$, and hence finite of length at most $3$.

Furthermore, we have
$$H^0(X_0,\tau_0^*(\mathcal{O}_{\mathbb{P}^5}(2))(-E_\Delta)\otimes
\mathcal{I}_{\gamma'})=H^0(X_0,L_0\otimes \mathcal{I}_{\gamma'})\cong
H^0(\PP^5,\mathcal{I}_\Gamma(2)),$$
$$H^0(X_0,\tau_0^*\mathcal{O}(5)(-E_\Delta)\otimes \mathcal{I}_{\gamma'})\cong
H^0(\PP^5,\mathcal{I}_\Gamma(5)).$$ It follows from the last
equality that (\ref{eq2407}) is equivalent to the fact that  $$
H^0(X_0,\tau_0^*\mathcal{O}(3))\cdot J'=
H^0(X_0,\tau_0^*\mathcal{O}(5)(-E_\Delta)\otimes
\mathcal{I}_{\gamma'})).$$
 Assume first that  $\gamma'$ is
curvilinear. It follows that by successively blowing-up at most
three points $x_1,\,x_2,\,x_3$ starting from $x_1\in X_0$, we get a
variety
$$\tau:X\rightarrow \PP^5,\,\tau_1:X\rightarrow X_0,$$ with three
exceptional divisors $E_i$ corresponding to the $x_i$'s and one
exceptional divisor $\tau_1^*E_\Delta$ over $E_\Delta$.  The $E_i$ are the pullbacks to $X$ of the exceptional divisor of the blow up at $x_i$ such that
the pull-backs $J''$ of the $J'$ gives rise to a base-point free
linear system of sections of
\begin{eqnarray}
\label{eqforL} L:=\tau^*\mathcal{O}(2)(-\tau_1^*E_\Delta-\sum_iE_i)
\end{eqnarray}
on $X$. Furthermore, we have
$$J''\subset H^0(X,L)\cong H^0(\PP^5,\mathcal{I}_\Gamma(2)),$$
$$H^0(X,\tau^*\mathcal{O}(5)(-\tau_1^*E_\Delta-\sum_iE_i))\cong
H^0(\PP^5,\mathcal{I}_\Gamma(5)).$$ We are thus reduced to prove
that the base-point free linear system $$J''\subset
H^0(X,\tau^*\mathcal{O}(2)(-\tau_1^*E_\Delta-\sum_iE_i))$$ generates
$H^0(X,\tau^*\mathcal{O}(5)(-\tau_1^*E_\Delta-\sum_iE_i))$. This is
done by a Koszul resolution argument. The Koszul resolution of the
surjective evaluation map
$$
J'' \otimes \mathcal{O}_X(-L)\rightarrow \mathcal{O}_X,$$
 gives us an exact complex with terms
$\bigwedge^iJ'' \otimes \mathcal{O}_X(-iL),\,0\leq i\leq5$. We twist this complex by
\begin{eqnarray}
\label{eqforLprime}L':=\tau^*\mathcal{O}(5)(-\tau_1^*E_\Delta-\sum_iE_i)
\end{eqnarray}
and the result then follows from the vanishing
\begin{eqnarray}
\label{eqvan2407}H^i(X,(-i-1)L+L')=0,\,i=1,\ldots,5. \end{eqnarray}
For $i=5$, we have by (\ref{eqforL}), (\ref{eqforLprime})
$$-6L+L'=\tau^*\mathcal{O}(-7)(5\sum_iE_i+5\tau_1^*E_\Delta),$$
while
$$K_X=\tau^*\mathcal{O}(-6)(4\sum_iE_i+3\tau_1^*E_\Delta).$$
Thus
$$H^5(X,-6L+L')=H^0(X,\tau^*\mathcal{O}(1)(-\sum_iE_i-2\tau_1^*E_\Delta))^*,$$
and the right hand side is  $0$.

For $i=4$, we have similarly
$$-5L+L'=\tau^*\mathcal{O}(-5)(4\sum_iE_i+4\tau_1^*E_\Delta),$$
hence
$$H^4(X,-5L+L')=H^1(X,\tau^*\mathcal{O}(-1)(-\tau_1^*E_\Delta))^*,$$
and the right hand side is $0$ since it is equal to
$H^1(\PP^5,\mathcal{I}_\Delta(-1))$.

For $i=3$, we have
$$-4L+L'=\tau^*\mathcal{O}(-3)(3\sum_iE_i+3\tau_1^*E_\Delta),$$
hence
$$H^3(X,-4L+L')=H^2(X,\tau^*\mathcal{O}(-3)(\sum_iE_i))^*.$$
Consider the strict transform $Y$ on $X$ of a general cubic fourfold whose pullback to $X_0$ contains $\gamma$.  Then  $\tau^*\mathcal{O}(-3)(\sum_iE_i)$ is the ideal sheaf of $Y$.  On the other hand $Y$ is regular so $H^1(Y,{\mathcal O}_Y)=0$, and hence $H^2(X,\tau^*\mathcal{O}(-3)(\sum_iE_i))=0$.

For $i=2$, we claim that
$$H^2(X,-3L+L')=H^2(X,\tau^*\mathcal{O}(-1)(2\sum_iE_i+2\tau_1^*E_\Delta)=0.$$
Consider the strict transform $Y$ on $X$ of a general hyperplane through $\Delta$ whose pullback to $X_0$ contains $\gamma$.  Then $Y$ is smooth and the multiplication by the form defining $Y$ fits in the exact sequence of sheaves
\[
 0\to \tau^*\mathcal{O}(-1)(2\sum_iE_i+2\tau_1^*E_\Delta) \to {\mathcal O}_X(\sum_iE_i+\tau_1^*E_\Delta)\to  {\mathcal O}_Y(\sum_iE_i+\tau_1^*E_\Delta)\to 0.
\]
But neither of the two invertible sheaves of exceptional divisors on the right have nonvanishing higher cohomology, so the claim follows.

For $i=1$, we get
$$H^1(X,-2L+L')=H^1(X,\tau^*\mathcal{O}(1)(\sum_iE_i+\tau_1^*E_\Delta)=0,$$
since $H^1(X,\tau^*\mathcal{O}(1))=H^1(X,E_1)=H^1(X,E_2)=H^1(X,E_3)=H^1(X,\tau_1^*E_\Delta)=0$.

When $\gamma$ is not curvilinear, and thus consists of one point
with noncurvilinear schematic
structure of length $3$, the argument is simpler:  Such a scheme $\gamma$ is the first order neighborhood of a point in a plane.
The image  $q_f(\Gamma)=q_f(\gamma)\cup q_f(\Delta)$ spans a $\PP^3$ and has by assumption no subscheme of length three contained in a line. But $q_f(\Delta)$ is a conic curve, while $q_f(\gamma)$ spans a plane that intersects this conic. Therefore there are lines that intersect the conic and $q_f(\gamma)$ in a subscheme of length $2$, a contradiction.
\end{proof}
To conclude the proof of Lemma \ref{propHgen}, we now show
\begin{sublemma}\label{le2307} Consider the cubic fourfolds
 $F=V(f)$ in the open dense subset $$\PP(S^3V)_{reg}\setminus (D_{rk3}\cup D_{copl}),$$ where $D_{copl}$ is the divisor introduced in Section \ref{sec1}.
 The subset of such fourfolds for which there exist a line $\Delta\subset \PP(V^*)$, and a hyperplane $J\subset
 I_\Delta(2)\cap Q_f$, such that the zero locus $\Gamma$ of $J$ has a subscheme residual to $\Delta$ of length $\geq 4$, has codimension $\geq 2$.
 
\end{sublemma}
 \begin{proof}  Note first that
 the scheme $\Gamma$ imposes at most $4$ conditions to $Q_f$, since
 $J\subset Q_f\cap I_\Gamma(2)$
 has codimension $4$ in $Q_f$.
 Therefore $q_f(\Gamma)$ is contained in the intersection of $X_f$ with a $\PP^3$, so, by Lemma \ref{lenew25juillet},
 the residual subscheme to $\Delta$ in $\Gamma$ is finite.
If it has length $\geq4$, we can  replace $\Gamma$ by a subscheme
$\Gamma'$ which is the union of $\Delta$ and a  residual scheme
$\gamma'$ of finite length $4$.  And, by Lemma \ref{q_f} (\ref{xxx}), we may assume $q_f(\Gamma')$ spans a $\PP^3$.
Note that $\Gamma'$, like $\Gamma$, is contained in an intersection of quadrics that is finite residual to the line $\Delta$,
so its intersection with a plane is either the line $\Delta$ or the union of the line $\Delta$ and one residual point, or it is a
scheme of finite length $\leq 4$.  Furthermore, the residual scheme
$\gamma'$ is not contained in another line $\Delta'$, since
otherwise the union of these two lines would be contained in
$\Gamma$. It follows that $\Gamma'$ imposes the maximal number of
conditions to the quadrics, namely $7$.
 Hence
  \begin{eqnarray}\label{eq09} \text{dim}\, (I_{\Gamma'}(2))=14,
   \end{eqnarray}
   and $J\subset I_{\Gamma'}(2)$ has dimension $11$.
   Since $q_f(\Gamma')$ spans a $\PP^3$,  the intersection $Q_f\cap I_{\Gamma'}(2)$ has dimension $11$, so it equals $J$.
 Now, $Q_f=P(f)^{\perp}$, so one concludes that
 \begin{eqnarray}\label{eq10aout}
 \text{dim}\,(P(f)\cap I_{\Gamma'}(2)^{\perp})=3,
 \end{eqnarray}
 where we recall
 that
 $P(f)$ is the space of partial derivatives of $f$.
The proof of Sublemma \ref{le2307} is done by a dimension count,
using (\ref{eq10aout}). We note that as we assumed that $f$ has no
partial derivative of rank $\leq3$, it has no net of partial derivatives
singular at a given point by Lemma \ref{lemmatouselater19aout}.
Thus, if $f$ satisfies (\ref{eq10aout}), the space
$I_{\Gamma'}(2)^{\perp}$ is not contained in the space of quadrics
singular at a given point. In particular, $\Gamma'$ must span $\PP(V)$.
This is equivalent to the vanishing
$H^1(\mathcal{I}_{\Gamma'}(1))=0$, which we assume from now on.

Equation (\ref{eq10aout})
determines a $3$-dimensional subspace $W\subset
V^*$, by
\[
W(f)=\{\partial_uf,\,u\in W\}=P(f)\cap I_{\Gamma'}(2)^{\perp}.
\]
 Given $W$ and $\Gamma'$, we define
$J^{\Gamma',W}\subset S^3V$ to be the linear space  of cubic forms
\[
J^{\Gamma',W}:=\{f\in S^3V| W(f)\subset I_{\Gamma'}(2)^{\perp}=Q_{\Gamma'}\}=(W\cdot I_{\Gamma'}(2))^{\perp}.
\]

The space $J^{\Gamma',W}$ contains the space
$J_{\Gamma'}:=I_{\Gamma'}(3)^\perp$ (which is generated by the cone
over the third Veronese embedding of $\Gamma'$) and the space
$S^3(W^{\perp})$.

Consider the subscheme
$$\Gamma'_W:=\PP(W^{\perp})\cap \Gamma'\subset\PP(V).$$
and assume first
that $\Gamma'_W=\emptyset$. In this
case, we claim that
\begin{eqnarray}\label{eqnpourKGammaprime} J^{{\Gamma'},W}=S^3(W^{\perp})\oplus J_{\Gamma'},
\end{eqnarray} so that ${\rm dim}\,J^{{\Gamma'},W}=18$.
Assuming the claim, we now observe  that elements
$$f\in J^{{\Gamma'},W}=S^3(W^{\perp})\oplus J_{\Gamma'}$$
fill-in, when the pair $(\Gamma', W)$ deforms, staying in general
position, the divisor $D_{copl}$ of Section \ref{sec1}. Indeed, the
general $\Gamma'$ is  the disjoint union of a line $\Delta=\PP(U)$
and 4 points $x_1,\ldots, x_4$. Then
$J_{\Gamma'}=S^3U+\langle x_1^3,\ldots,x_4^3 \rangle $ and thus $f\in
S^3(W^{\perp})\oplus J_{\Gamma'}$ belongs to $S^3(W^{\perp})+
S^3U+\langle x_1^3,\ldots,x_4^3 \rangle $. The component of $f$ lying in
$S^3(W^{\perp})$ is the sum of $4$ cubes of coplanar linear forms, and
the component of $f$ lying in $S^3U$ is the sum of $2$ cubes. Thus
$f$ is the sum of $10$ cubes of linear forms, $4$ of which are
coplanar.

In order to prove formula (\ref{eqnpourKGammaprime}), we dualize it
and note that it is equivalent to the equality
\begin{eqnarray}\label{eqnpourWI21aout}W\cdot I_{\Gamma'}(2)=(W\cdot S^2V^*)\cap I_{\Gamma'}(3).
\end{eqnarray}
The right hand side is equal to $I_{\Gamma'\cup \PP(W^{\perp})}(3)$.
As $\Gamma'\cap\PP(W^{\perp})=\emptyset$, the Koszul resolution of
the ideal sheaf $\mathcal{I}_{\PP(W^{\perp})}$ remains exact after
tensoring by $\mathcal{I}_{\Gamma'}$, which gives  the following
resolution of $\mathcal{I}_{\Gamma'\cup \PP(W^{\perp})}$:
$$0\rightarrow
\bigwedge^3W\otimes\mathcal{I}_{\Gamma'}(-3)\rightarrow
\bigwedge^2W\otimes\mathcal{I}_{\Gamma'}(-2)\rightarrow
W\otimes\mathcal{I}_{\Gamma'}(-1)\rightarrow
\mathcal{I}_{\Gamma'\cup \PP(W^{\perp})}\rightarrow0.$$ Twisting
with $\mathcal{O}(3)$ and applying the  vanishings $H^1(\mathcal{I}_{\Gamma'}(1))=0$ and $H^2(\mathcal{I}_{\Gamma'})=0$, we  get the desired equality
$W\cdot I_{\Gamma'}(2)=I_{\Gamma'\cup \PP(W^{\perp})}(3)$.

To conclude the proof of Sublemma \ref{le2307}, it only remains to
prove the following claim:
\begin{claim} \label{claim21aout} The set of cubic fourfolds in the open set $\PP^{55}_{reg}\setminus
D_{rk3}:=\PP(S^3V)_{reg}\setminus
D_{rk3}$ satisfying (\ref{eq10aout}) for a pair $(W,\Gamma')$ with
$\Gamma'_W=\Gamma'\cap\PP(W^{\perp})\not=\emptyset$ has codimension $\geq2$.
\end{claim}
\end{proof}
\begin{proof}[Proof of Claim \ref{claim21aout}]  Recall, from above, that $\Gamma'$ contains a line $\Delta$ and spans $\PP^5$.  Also, since $q_f(\Gamma')$  spans a $\PP^3$ and contains no subscheme of length $3$ in a line, every component of $\Gamma'$ that is not supported on $\Delta$ is curvilinear.
Consider the intersection $\Gamma'_W=\Gamma'\cap\PP(W^{\perp})$.

If $\Gamma'_W$ contains $\Delta$, then, since it is the intersection of quadrics and is finite residual to $\Delta$, the residual scheme to $\Delta$ in $\Gamma'_W$ is at most a point.

If $\Gamma'_W$ intersects $\Delta$ only in a point $x$, then $\Gamma'_W\cup \Delta$ spans at most a $\PP^3$, so $\Gamma'$ has a scheme of length at least $2$ residual to $\Gamma'_W\cup \Delta$.  By Sublemma \ref{isoresiduels}, the scheme $\Gamma'_W$ has a component of length at most $2$ supported on $x$ and a residual closed point $x'$.

If  $\Gamma'_W$ does not intersect the line $\Delta$, then $\Gamma'_W$ is curvilinear and has length at most $3$.

We observe that in each of the listed situations, if
$X,Y\in W$ are generically chosen, and $\PP^3_{X,Y}\supseteq
\PP(W^{\perp})$ is defined by $X$ and $Y$, we have
$$\Gamma'\cap\PP^3_{X,Y}=\Gamma'\cap\PP(W^{\perp})=\Gamma'_W.$$

We want to estimate the dimension of $J^{\Gamma',W}=(W\cdot I_{\Gamma'}(2))^{\perp}$, or equivalently of $W\cdot I_{\Gamma'}(2)$, since
$$\text{dim}\,J^{\Gamma',W}=56-\text{dim}\,(W\cdot I_{\Gamma'}(2)).$$

We consider the exact sequence
\[
 0\to \langle X,Y \rangle \cdot
I_{\Gamma'}(2) \to W\cdot I_{\Gamma'}(2)\to  W\cdot I_{\Gamma'}(2)_{\mid \PP^3_{X,Y}}\to 0.
\]
and observe that $\text{dim}\,W\cdot I_{\Gamma'}(2)_{\mid \PP^3_{X,Y}}= \text{dim}\,I_{\Gamma'}(2)_{\mid \PP^3_{X,Y}}$.
Therefore

\begin{eqnarray}\label{eqn12}\text{dim}\,(W\cdot I_{\Gamma'}(2))=\text{dim}\,(\langle X,Y \rangle \cdot
I_{\Gamma'}(2))+\text{dim}\,I_{\Gamma'}(2)_{\mid \PP^3_{X,Y}}.
\end{eqnarray}

Furthermore, consider the space of linear forms $[I_{\Gamma'}(2):\langle X,Y \rangle ]\subset V.$ 
Multiplication by the matrix $(X, -Y)$ and $(Y,X)^t$ respectively defines an exact sequence

\[
 0\to [I_{\Gamma'}(2):\langle X,Y \rangle ] \to I_{\Gamma'}(2)\oplus I_{\Gamma'}(2)\to  \langle X,Y \rangle \cdot
I_{\Gamma'}(2)\to 0.
\]
From this sequence, and the fact (\ref{eq09}) that $\text{dim}\,I_{\Gamma'}(2)=14$, we get
\[
\text{dim}\,(\langle X,Y \rangle \cdot
I_{\Gamma'}(2))=2\,\text{dim}\,I_{\Gamma'}(2)-\text{dim}\,[I_{\Gamma'}(2):\langle X,Y \rangle ]= 28-\text{dim}\,[I_{\Gamma'}(2):\langle X,Y \rangle ].
\]
 Putting this equality
together with the equation (\ref{eqn12}) we get:
$$\text{dim}\,J^{\Gamma',W}=28+\text{dim}\,[I_{\Gamma'}(2):\langle X,Y \rangle ]-\text{dim}\,I_{\Gamma'}(2)_{\mid
\PP^3_{X,Y}}.$$ We make now a case-by-case analysis. Recall that if
the scheme $\Gamma'_W$    has finite length, this length is $\leq3$
and if it contains the line $\Delta$, it contains at most one
reduced residual point.
  \begin{enumerate}

  \item If $\Gamma'_W=[l]$ is a  reduced point on $\Delta=\PP(U)$, which is not the support of an embedded point,
   then $\text{dim}\,[I_{\Gamma'}(2):\langle X,Y \rangle ]=0$
  and $\text{dim}\,I_{\Gamma'}(2)_{\mid
\PP^3_{X,Y}}=9$, so we get  ${\rm dim}\,J^{{\Gamma'},W}=19$. The
parameter space for such  $(W,\Gamma)'$s has dimension $7+28=35$, so
the subset of $\PP^{55}_{reg}$ satisfying equation (\ref{eq10aout})
with this condition on $(W,\Gamma)$ has dimension $\leq 35+18=53$.
\item If $\Gamma'_W=[l]$ is a reduced point on $\Delta=\PP(U)$, which is the support of an embedded point,
 then $\text{dim}\,[I_{\Gamma'}(2):\langle X,Y \rangle ]=1$
  and $\text{dim}\,I_{\Gamma'}(2)_{\mid
\PP^3_{X,Y}}=9$. Thus ${\rm dim}\,J^{{\Gamma'},W}=20$. As $\Gamma'$
has an embedded point on $\Delta$, the parameter space for $\Gamma'$
has dimension $27$, so  the parameter space for such  $(W,\Gamma)'$s
has dimension $7+27=34$. Thus the subset of $\PP^{55}_{reg}$
satisfying equation (\ref{eq10aout}) with this condition on
$(W,\Gamma)$ has dimension $\leq 34+19=53$.
  \item If $\Gamma'_W=[l]$ is a reduced  point not in  $\Delta$, then $\text{dim}\,[I_{\Gamma'}(2):\langle X,Y \rangle ]=1$
  and $\text{dim}\,I_{\Gamma'}(2)_{\mid
\PP^3_{X,Y}}=9$, so we get  ${\rm dim}\,J^{{\Gamma'},W}=20$. The
parameter space for such  $(W,\Gamma)'$s has dimension $6+28=34$, so
the subset of $\PP^{55}_{reg}$ satisfying equation (\ref{eq10aout})
with this condition on $(W,\Gamma)$ has dimension $\leq 34+19=53$.
 \item If $\Gamma'_W$  is a subscheme of length $2$ that  intersects
 $\Delta$ in one point, which is not the support of an embedded point,
   then $\text{dim}\,[I_{\Gamma'}(2):\langle X,Y \rangle ]=1$
  and $\text{dim}\,I_{\Gamma'}(2)_{\mid
\PP^3_{X,Y}}=8$, so we get  ${\rm dim}\,J^{{\Gamma'},W}=21$. The
parameter space for such  $(W,\Gamma)'$s has dimension $4+28=32$, so
the subset of $\PP^{55}_{reg}$ satisfying equation (\ref{eq10aout})
with this condition on $(W,\Gamma)$ has dimension $\leq 32+20=52$.

\item If $\Gamma'_W$  is a subscheme of length $2$ that  intersects
 $\Delta$ in one point, which is the support of an embedded point,
   then $\text{dim}\,[I_{\Gamma'}(2):\langle X,Y \rangle ]=2$
  and $\text{dim}\,I_{\Gamma'}(2)_{\mid
\PP^3_{X,Y}}=8$, so we get  ${\rm dim}\,J^{{\Gamma'},W}=22$. The
parameter space for such  $(W,\Gamma)'$s has dimension $3+27=30$, so
the subset of $\PP^{55}_{reg}$ satisfying equation (\ref{eq10aout})
with this condition on $(W,\Gamma)$ has dimension $\leq 30+21=51$.
 \item If $\Gamma'_W=z_2$  is a subscheme of length $2$ that does not intersect $\Delta$,
 then $\text{dim}\,[I_{\Gamma'}(2):\langle X,Y \rangle ]=2$
  and $\text{dim}\,I_{\Gamma'}(2)_{\mid
\PP^3_{X,Y}}=8$, so we get  ${\rm dim}\,J^{{\Gamma'},W}=22$. The
parameter space for such  $(W,\Gamma)'$s has dimension $3+28=31$, so
the subset of $\PP^{55}_{reg}$ satisfying equation (\ref{eq10aout})
with this condition on $(W,\Gamma)$ has dimension $\leq 31+21=52$.

        \item If $\Gamma'_W=\Delta$, then
        then $\text{dim}\,[I_{\Gamma'}(2):\langle X,Y \rangle ]=2$
  and $\text{dim}\,I_{\Gamma'}(2)_{\mid
\PP^3_{X,Y}}=7$, so we get  ${\rm dim}\,J^{{\Gamma'},W}=23$. The
parameter space for such  $(W,\Gamma)'$s has dimension $3+28=31$, so
the subset of $\PP^{55}_{reg}$ satisfying equation (\ref{eq10aout})
with this condition on $(W,\Gamma)$ has dimension $\leq 31+22=53$.

 \item If $\Gamma'_W$ is a subscheme of length $3$ that does not intersect $\Delta$,
  then $$\text{dim}\,[I_{\Gamma'}(2):\langle X,Y \rangle ]=3$$
  and $\text{dim}\,I_{\Gamma'}(2)_{\mid
\PP^3_{X,Y}}=7$, so we get  ${\rm dim}\,J^{{\Gamma'},W}=24$. The
parameter space for such  $(W,\Gamma)'$s has dimension $28$, so the
subset of $\PP^{55}_{reg}$ satisfying equation (\ref{eq10aout}) with
this condition on $(W,\Gamma)$ has dimension $\leq 23+28=51$.
 \item If $\Gamma'_W$ is a subscheme of length $3$ that  intersects $\Delta$ in a point
 $[l]$, which is the support of an embedded point,
  then $\text{dim}\,[I_{\Gamma'}(2):\langle X,Y \rangle ]=3$
  and $\text{dim}\,I_{\Gamma'}(2)_{\mid
\PP^3_{X,Y}}=7$, so we get  ${\rm dim}\,J^{{\Gamma'},W}=24$. The
parameter space for such  $(W,\Gamma)'$s has dimension $27$, so the
subset of $\PP^{55}_{reg}$ satisfying equation (\ref{eq10aout}) with
this condition on $(W,\Gamma)$ has dimension $\leq 27+23=50$.

\item If $\Gamma'_W$ is a subscheme of length $3$ that  intersects $\Delta$ in a point
 $[l]$, which is not the support of an embedded point,
  then $\text{dim}\,[I_{\Gamma'}(2):\langle X,Y \rangle ]=2$
  and $\text{dim}\,I_{\Gamma'}(2)_{\mid
\PP^3_{X,Y}}=7$, so we get  ${\rm dim}\,J^{{\Gamma'},W}=23$. The
parameter space for such  $(W,\Gamma)'$s has dimension $1+28=29$, so
the subset of $\PP^{55}_{reg}$ satisfying equation (\ref{eq10aout})
with this condition on $(W,\Gamma)$ has dimension $\leq 22+29=51$.
  \item If $\Gamma'_W$ is the union of the line $\Delta$ and an embedded point, then $\text{dim}\,[I_{\Gamma'}(2):\langle X,Y \rangle ]=3$
  and $\text{dim}\,I_{\Gamma'}(2)_{\mid
\PP^3_{X,Y}}=6$, so we get  ${\rm dim}\,J^{{\Gamma'},W}=25$. The
parameter space for such  $(W,\Gamma)'$s has dimension $28$, so the
subset of $\PP^{55}_{reg}$ satisfying equation (\ref{eq10aout}) with
this condition on $(W,\Gamma)$ has dimension $\leq 24+28=52$.
\end{enumerate}

   This proves the claim.

\end{proof}
The proof of  Lemma \ref{propHgen}, hence also of Proposition
\ref{allapolar}, is finished.
\end{proof}

\subsection{Syzygies}

Recall that the cactus rank of a cubic fourfold $F=V(f)$ is the minimal length of an apolar subscheme (Definition \ref{defcactusrank}).
We consider the syzygies of the ideal $I_f$, and give below a partial characterization of cubic fourfolds of cactus
rank $<10$, which we will  use  to prove  Proposition
\ref{exampleveronese} in the next section.

For a cubic fourfold
$F\subset \PP(V^*)$, let $V_{10}(F)\subset \PP(V)$ be the union of
subschemes of length $10$ which are apolar to $F$. We shall show, in Lemma \ref{leautiliser}, that when $F$ is general and of cactus rank $10$, then $V_{10}(F)$ is a hypersurface of degree $9$.
 As suggested to us by Hans Christian von Bothmer, to find the equation of $V_{10}(F)$, when it is a hypersurface,  we study the syzygies of the apolar ideal  $I_f$ and compare it with syzygies of the ideal of subschemes of length $9$ and $10$.

We are interested in the graded Betti numbers for the minimal free
resolution of the ideal $I_f$ for a general $f$, and for the ideal
of a general set of $9$ and $10$ points.
\begin{examp}\label{exa} The Betti numbers in the following examples have been computed with {\it Macaulay2} \cite{MAC2}.
\begin{enumerate}
\item  Let $f\in \CC[x_0,...,x_5]$ be the cubic form

\begin{align*}
f&= 2x_1^2x_2-2x_0x_2^2-2x_1^2x_3-2x_3^2x_4-x_0x_1x_5+2x_1x_2x_5\\
&+x_2^2x_5+x_2x_3x_5+3x_1x_4x_5+x_4^2x_5+3x_0x_5^2+x_3x_5^2 \\
    \end{align*}
Then the resolution of $I_f$ has Betti numbers:
  \[ \begin{array}{ccccccc}
 1 & - & - & - & - & -& - \\
 - & 15 & 35 & 21& - & -& - \\
 - & - & - & 21& 35 & 15& - \\
 - & - & - & - & - & -& 1
 \end{array}.
 \]

\item Let $Z_6$ be the $6$ coordinate points in $\PP^5$, then the resolution of the ideal of the $9$ points 
$$Z_9=Z_6\cup \{
(1:1:1:1:0:0), (0:0:1:-1:-1:1),(1:-1:0:0:1:1)\}$$ in $\PP^5$ has Betti numbers
\[ \begin{array}{ccccccc}
 1 & - & - & - & - & -\\
 - & 12 & 25 & 15& - & -\\
 - & - & - & 6& 10 & 3\\
 \end{array},
 \]

\item The resolution of the ideal of the $10$ points $Z_9\cup \{(0:1:-1:-1:0:1)\}$
in $\PP^5$ has Betti numbers
\[ \begin{array}{ccccccc}
 1 & - & - & - & - & -\\
 - & 11 & 20 & 5& - & -\\
 - & - & - & 16& 15 & 4\\
 \end{array}.
 \]

\end{enumerate}
\end{examp}
\begin{remark}\label{bettinu}
{\rm The graded Betti
numbers in the three resolutions  of Example  \ref{exa} are clearly minimal for apolar ideals of cubic forms and for ideals of $9$ (resp. $10$) points in $\PP^5$.
By semicontinuity we conclude that the Betti numbers are the same as in these examples for a general cubic form,
and for the ideal of $9$ (resp. $10$) general points in $\PP^5$.}
\end{remark}

 \begin{lemma}\label{rank9}  If $f$ is a cubic form with no partials of
  rank $\leq 3$, then $f$ has  cactus rank $\geq 9$.
 If furthermore the minimal free resolution of the apolar ideal $I_f$ has Betti numbers
  \[ \begin{array}{ccccccc}
 1 & - & - & - & - & -& - \\
 - & 15 & 35 & 21& - & -& - \\
 - & - & - & 21& 35 & 15& - \\
 - & - & - & - & - & -& 1
 \end{array}
 \]
 and $f$ has cactus rank $9$, then the $(35\times 21)$-matrix $M_2$ of linear second order syzygies has generic rank at most
 $20$. In other words, if  $f$ has no partial derivative of rank
 $\leq3$, the apolar ideal $I_f$ has Betti numbers as above and the matrix $M_2$ has generic rank $21$, then
 $f$ has cactus rank $10$.
 \end{lemma}
 \begin{proof}  Since $f$ has no partial derivatives of rank $\leq 3$, the map $q_f: \PP(V)\to X_f$ is a
 smooth embedding and $X_f$ has no trisecant lines, by Lemma \ref{q_f}.  Let $Z$ be an apolar subscheme of length at most $8$.
Since
  $I_Z(2)\subset Q_f=I_f(2)$ and $Q_f\subset S^2V$ has codimension $6$, the rank of the restriction map
  $Q_f\rightarrow H^0(\mathcal{O}_Z(2))$ is at most $2$. Hence
  $Z_f=q_f(Z)$ is contained in a line, and $X_f$ would have a trisecant line, a contradiction. Therefore $f$ has cactus rank at least $9$.

Assume next, $f$ has cactus rank $9$ computed by an apolar subscheme $Z\subset\PP(V)$ that consists of $9$ general points. We consider the Gale transform of $Z$ (cf. \cite{EP}).
 The Gale transform $Z'$ of $Z$ is a set of $9$ points in a plane, and $Z'$ is general since $Z$  is general.  In particular we may assume that  $Z'$  lies on a unique smooth cubic curve.  By \cite[Corollary 3.2]{EP}, the set of $9$ points $Z$ itself lies on this
curve reembedded as an elliptic sextic curve $E_Z$ in $\PP(V)$. The
Betti numbers of the minimal free resolution of the ideal of $E_Z$
are
 \[ \begin{array}{ccccccc}
 1 & - & - & - &  -\\
 - & 9 & 16 & 9 & -\\
 - & - & - & -& 1\\
 \end{array}.
 \]
Since $I_{E_Z}\subset I_Z$ and $I_Z\subset I_f$, by assumption, we get that $I_{E_Z}\subset I_f$, i.e. the elliptic sextic curve $E_Z$ is apolar to $f$.  The inclusion of the resolution of $I_{E_Z}$ in the resolution of $I_f$ displays a third order syzygy of the ideal $I_{E_Z}$ that is a third order syzygy for the linear strand of the resolution of the ideal $I_{f}$. In the resolution of $I_f$ the matrix $M_2$ therefore has generic rank at most $20$.

It remains to consider any cubic  form $f$ of cactus rank $9$ and no partials of rank at most $3$.  Let $Z$ be a length $9$ subscheme apolar to $f$.   
Now, by Remark \ref{BC}, the scheme $Z$ is locally Gorenstein and the limit of smooth schemes of length $9$, the form $f$ is likewise a limit of forms of cactus rank $9$ with a smooth apolar scheme of length $9$.  Therefore, by the previous argument, the matrix $M_2$ in the resolution of $I_f$ is the limit of matrices of generic rank at most $20$, so $M_2$ also has generic rank at most $20$.

\end{proof}

We analyze further the syzygies of elliptic normal sextic curves, to find the locus in $\PP(V)$ where the matrix $M_2$ in the resolution of $I_f$  drops rank.

First, an elliptic normal sextic curve $E$ lies in a smooth Veronese surface:  any of the four linear systems $|D|$ of degree $3$ on $E$ such that $|2D|$ is the linear system of hyperplane sections of $E\subset \PP(V)$, is the linear system of conic sections of $E$ in a smooth Veronese surface in $\PP(V)$.

\begin{lemma}\label{second order lin} Let $E$ be an elliptic normal sextic curve in $\PP^5$ and let $p\in \PP(V)\setminus E$.
 Then the ideal of $E\cup \{p\}$ has a unique second order linear syzygy that vanishes at $p$.

  If, in addition, $p$ is not contained in the secant variety of any of the four Veronese surfaces containing $E$, then the syzygy has rank $5$ and no ideal strictly contained in the ideal of $E\cup \{p\}$ has this syzygy.

 If $p\in \Sigma\setminus E$, where $\Sigma$ is a Veronese surface that contains $E$, then the second order linear syzygy that vanishes at $p$ is a syzygy for $I_\Sigma$, but no ideal strictly contained in $I_\Sigma$.
\end{lemma}
\begin{proof}
Let $p$ be in $\PP(V)\setminus E$.    The  minimal free resolution of $I_E$ is symmetric with Betti numbers
 \[ \begin{array}{ccccccc}
 1 & - & - & - &  -\\
 - & 9 & 16 & 9 & -\\
 - & - & - & -& 1\\
 \end{array}.
 \]
 The third order syzygy is therefore nonzero at the point $p$ outside $E$, and defines a unique  second order syzygy that vanishes at $p$.

 This syzygy is a syzygy among at most $5$ first order linear syzygies that also vanishes at $p$,
 and finally, these first order syzygies are linear syzygies among quadrics in the ideal of $E$ that vanish at $p$.
 Therefore the ideal of $I_{E\cup \{p\}}$ has a second order linear syzygy vanishing at $p$.

The secant variety of $E$ is the intersection of a pencil of determinental cubic hypersurfaces that are singular along $E$ (see~\cite{Fisher} {Theorem~1.3, Lemma~2.9}).
In fact, these hypersurfaces are defined by determinants of $(3\times 3)$-matrices with linear entries that have rank one along $E$.  Since $E$ is smooth, all the linear entries are nonzero.   Four of the cubic  hypersurfaces are secant varieties of Veronese surfaces that contain $E$.   If $p$ is not on any of these four hypersurfaces, then $p$ is not on the secant variety of $E$, and there is a unique cubic determinental hypersurface $Y$ that is singular along $E$ and contains $p$. We may assume that this hypersurface is defined by the determinant of the $(3\times 3)$-matrix  $A=\left(\begin{matrix} a_0&a_1&a_2\\ a_3&a_4& a_5\\ a_6&a_7&a_8\\ \end{matrix}\right)$ with linear entries $a_i$.   Since $p$ is not on any of the four secant varieties of a Veronese surface containing $E$, the matrix $A$ is not symmetric, so $p$ is in a unique plane in each of the two nets of planes in $Y$.  In particular we may assume that $p=V(a_0,a_1, a_2, a_4, a_7)$. Then the ideal of $E\cup \{p\}$ is generated by all the $(2\times 2)$-minors of $A$ except $a_3a_8-a_6a_5$, i.e. the quadrics $I_{E\cup \{p\}}=(a_1a_3-a_0a_4, a_2a_3-a_0a_5 ,a_2a_4-a_1a_5, a_1a_6-a_0a_7, a_2a_6-a_0a_8,a_4a_6-a_3a_7 , a_2a_7-a_1a_8, a_5a_7-a_4a_8)$.  These quadrics have the following matrices of first and second order  linear  syzygies that vanish at $p$:
 \[ S_1= \left(\begin{array}{ccccccc}
 -a_2 & 0 & a_7 & 0 &0 \\
 a_1& 0 & a_0 & a_7& 0\\
 -a_0 & 0 & 0 & 0&-a_7 \\
  0& -a_2 & -a_4 & 0& 0 \\
  0& a_1& 0 & -a_4& 0 \\
    0& 0 & a_1 & a_2& 0 \\
     0& -a_0 & 0 & 0& a_4 \\
      0 & 0 & 0 & a_0& -a_1 \\
 \end{array}\right) \quad {\rm and}\quad S_2=
 \left( \begin{array}{ccccccc}
 -a_7  \\
 a_4\\
-a_2  \\
  a_1 \\
   a_0 \\
     \end{array}\right),
 \]
 i.e.
 \[
 I_{E\cup \{p\},2}\cdot S_1=S_1\cdot S_2=0.
 \]
 The rows of the matrix $S_1$ have no constant syzygies, so there are no ideal properly contained in $I_{E\cup \{p\}}$ with the second order linear syzygy, $S_2$, among its quadrics.  

 Now, if the matrix above is symmetric, i.e. $a_1=a_3,a_2=a_6,a_5=a_7$, it has rank one along a Veronese surface $\Sigma$.  
Therefore, for each point $p\in \Sigma$,  there is an inclusion $I_\Sigma\subset I_{E\cup \{p\}}$.  The quadrics in $I_{E\cup \{p\}}$ in the non-symmetric case reduces to the quadrics in $I_\Sigma$.  If $p=V(a_0,a_1, a_2, a_4, a_7)\in \Sigma$, the above displayed second order syzygy remains a second order linear syzygy among the quadrics in $I_\Sigma$, and as in the non-symmetric case, no proper ideal contained in $I_\Sigma$ has this second order linear syzygy. 

\end{proof}

Assume now that $Z$ is a set of $10$ points that is apolar to a cubic fourfold $F$ of rank $10$, and that a subset $Z_0\subset Z$ of $9$ points lies in an elliptic normal sextic curve $E$.  By Lemma \ref{second order lin}, the ideal of $E\cup Z$, and hence also $I_Z$, has a second order syzygy that vanishes in the point $p=Z\setminus Z_0$ so the matrix $M_2$ in the resolution of $I_f$ has rank at most $20$ at $p$.

Furthermore, assume that the set $Z_0\subset E$ of $9$ points  in $\PP(V)$ is
  the base locus of a pencil of elliptic sextic curves $\{E_\lambda\}$ on a Veronese surface, and that the point $p$ is not contained in the secant variety of this Veronese surface.  Then, for a general curve $E_\lambda$ in the pencil, the second order linear syzygy for the ideal of $E_\lambda\cup Z$ determines the curve. Therefore there is a pencil of second order syzygies for $I_Z$ that vanishes at $p=Z\setminus Z_0$.  Hence, we conclude that  the matrix $M_2$ in the resolution of $I_f$ has rank at most $19$ at $p$.

Let $F$ be a cubic fourfold defined by a form $f$ of rank $10$ and consider the
incidence
\[
I_{VSP}=\{(p,[Z])|p\in Z\}\subset \PP(V)\times VSP(F,10).
\]
Then, by definition,  $V_{10}(F)\subset \PP(V)$ is the image of
$I_{VSP}$ under  the first projection $I_{VSP}\to  \PP(V)$.   

 \begin{corollary}\label{second order} Let $f$ be a cubic form of rank $10$ with no partial
 derivatives of rank $\leq 3$ and Betti numbers
  \[ \begin{array}{ccccccc}
 1 & - & - & - & - & -& - \\
 - & 15 & 35 & 21& - & -& - \\
 - & - & - & 21& 35 & 15& - \\
 - & - & - & - & - & -& 1
 \end{array}
 \]
 for the apolar ideal $I_f$.  Assume that there exist a set $Z$ of $10$ points apolar to $f$, and that a subset $Z_0\subset Z$ of $9$ points lie on an elliptic sextic curve $E_{Z_0}$, while the point $p=Z\setminus Z_0$ does not lie on the secant variety of any Veronese surface that contains $E_{Z_0}$.
Then the $(35\times 21)$-matrix $M_2$ of linear second order
syzygies has rank at most $20$ along $V_{10}(F)$. Furthermore, $M_2$
has rank at most $19$ at every point $p\in Z\subset \PP(V)$ such that the subscheme $Z_0$ 
is contained in a pencil of elliptic sextic curves on a Veronese surface.

 \end{corollary}
 \begin{proof}   The first condition on $Z_0$ and $Z$ is clearly an open condition in $VSP(F,10)$, so $M_2$  
  has rank at most $20$  along a Zariski open set of $V_{10}(F)$, hence everywhere along $V_{10}(F)$.
  Similarly the second condition  on $Z_0$ and $Z$ is open among sets of points $Z$ such that the subset $Z_0$ is contained in a pencil of 
elliptic sextic curves on a Veronese surface, so the second part of the Corollary follows.

\end{proof}

\begin{lemma}\label{leautiliser} (von Bothmer \cite{Bo})
Let $f$ be a cubic form whose apolar ideal $I_f$ has a minimal free resolution with Betti numbers
 \[ \begin{array}{ccccccc}
 1 & - & - & - & - & -& - \\
 - & 15 & 35 & 21& - & -& - \\
 - & - & - & 21& 35 & 15& - \\
 - & - & - & - & - & -& 1
 \end{array}.
 \]
Then the $(35\times 21)$-matrix $M_2$ of linear second order
syzygies has rank at most $20$ either on all of $\PP(V)$ or on a hypersurface
$Y_F$.   In the second case, $V_{10}(F)$ is equal to
$Y_F$ and  has degree $9$ if  $M_2$ has rank $20$ at a general point of $Y_F$.
\end{lemma}
\begin{proof}
Consider the linear strand of the resolution of $I_f$ with Betti numbers
 \[ \begin{array}{ccccccc}
 1 & - & - & - & -  \\
 - & 15 & 35 & 21& -
 \end{array}
 \]
 evaluated at a general point.
The first map has kernel of dimension $14$.  Therefore the corank of the
third map $\varphi_{M_2}$ is at least $14$. If the linear strand is
exact at a general point, then the rank of the third map drops along
a hypersurface.   We compute the degree of this hypersurface by restricting the linear strand to a general line $L\subset \PP(V)$.
This restriction of the linear strand is a complex
\[
0\leftarrow {\mathcal O}_{L}\leftarrow 15{\mathcal O}_{L}(-2)
 \leftarrow 35{\mathcal O}_{L}(-3)\leftarrow 21{\mathcal O}_{L}(-4)\leftarrow 0
\]
that is exact, except at $35{\mathcal O}_{L}(-3)$.
The kernel of the first map is a vector bundle $E_1$ of rank 14 and first Chern class $c_1(E_1)=-30$ on $L$.  Therefore, the second map factors into a surjective map
$E_1\leftarrow 35{\mathcal O}_{L}(-3)$ with kernel a vector bundle $E_2$ of rank 21 with first Chern class $c_1(E_2)=35\cdot (-3)-(-30)=-75$ on $L$.
The third map of the linear strand, defined by the restriction of  $\varphi_{M_2}$ to $L$, factors through a vector bundle map $E_2 \leftarrow 21{\mathcal O}_{L}(-4)$
between two bundles of rank $21$.  The determinant of this bundle map, since it is assumed to be nonzero, defines a divisor whose class is the difference of the first Chern classes of the two bundles, i.e. of degree $-75+21\cdot (-4)=9$ on $L$.  So $\varphi_{M_2}$ either has rank at most $20$ on all of $\PP(V)$ or it has rank at most $20$ on a hypersurface of degree $9$.

For the last statement, we already proved in Corollary \ref{second
order} that the hypersurface $V_{10}(F)\subset \PP(V)$ is contained
in the determinental hypersurface $Y_F$ of points where $M_2$ has
rank at most $20$.  On the other hand, one can exhibit $F$ for which
$Y_F$ is irreducible (cf. Proposition \ref{exampleveronese}).
Hence for such an $F$, $V_{10}(F)$ must be equal to $Y_F$, which
implies the same result for any $F$.
\end{proof}
\begin{remark}{\rm  The general cubic fourfold $F$ in the divisor $D_{IR}$ has rank $10$, while $V_{10}(F)$ is a Pfaffian cubic hypersurface (cf. \cite[Lemma 3.9 and Proposition 3.15]{IR}).  In this case the $(35\times 21)$-matrix $M_2$ has rank $18$ at a general point of $V_{10}(F)$.
}
\end{remark}

\begin{lemma} \label{lesingVdixdet} Let  $f\in S^3V$ be a cubic form 
whose apolar ideal $I_f$ has a minimal free resolution with Betti numbers
 \[ \begin{array}{ccccccc}
 1 & - & - & - & - & -& - \\
 - & 15 & 35 & 21& - & -& - \\
 - & - & - & 21& 35 & 15& - \\
 - & - & - & - & - & -& 1
 \end{array}.
 \]
 If the $(35\times 21)$-matrix $M_2$ of linear second order
syzygies has rank $21$ at a general point and rank $20$ at some point,
  then $V_{10}(F)$ is singular along the set of points $[l]\in \PP(V)$ for
  which $f-l^3$ has rank $9$ and the matrix $M_2$ has rank at most $19$, in particular at the points $[l]$
for which $f-l^3$ is apolar to a pencil of elliptic
normal sextic curves on a Veronese surface.
\end{lemma}
\begin{proof}
Indeed, by Lemma \ref{leautiliser}, the $(35\times 21)$-matrix $M_2$ has rank at most $20$ along the determinantal hypersurface $V_{10}(F)=Y_F$.  Since it has rank $20$ at some point, it must have rank $20$ at a general point of $V_{10}(F)$, while $V_{10}(F)$  is singular where the rank is at most $19$.  The lemma therefore follows from Corollary \ref{second order}.
\end{proof}
\begin{remark} \label{rank19alongsurface} {\rm We have computed with {\it Macaulay2} \cite{MAC2}  for certain cubic forms $f$, that $V_{10}(F)$ is a
hypersurface of degree $9$ whose singular locus is a surface of
degree $140$ that coincides with the locus where $M_2$ has rank at
most $19$. Therefore we conjecture that this holds for a general
$f$.}\end{remark}

\begin{lemma}\label{twonines} Let $f$ be a cubic form of rank $9$ and assume that there is a $9$-tuple of points apolar to $f$ that is a divisor $D$ on an elliptic sextic curve and that $2D$ is not a cubic hypersurface divisor on the curve.  Then there are exactly two subschemes of length $9$ that are apolar to $f$.
\end{lemma}
\begin{proof} 
Let $D=\{p_1,...,p_9\}$ be a set of points on an elliptic sextic curve $E$ apolar to $f$.  Then the Gale transform of the
points $D$ are $9$ points $D'\subset\PP^2$.

The Gale transform  (cf. \cite{EP}) reembeds $E$ as a cubic
curve $E'$ through the points $D'$ in $\PP^2$, such that the lines in the plane intersect $E'$ in divisors $H_3$ equivalent to $D-H_6$, where $H_6$ is a hyperplane divisor on $E$ in $\PP^5$.
Since $2D$ is not equivalent to $3H_6$, the cubic divisor in the plane $3H_3=3D-3H_6=D+(2D-3H_6)$ is not equivalent to $D$.
Therefore $D'$ lies on a unique cubic curve in the plane, and likewise $E$ is the unique elliptic sextic curve in $\PP^5$ through $D$.
The curve $E$ is apolar to $f$, and we claim that any $9$-tuple of points
apolar to $f$ lies on this curve.

By Terracini's Lemma, (cf. \cite{Ter}, \cite{Zak}), the tangent space to the $8$-th secant variety
of the $3$-uple embedding $W_3\subset\Pn {55}$ of $\Pn 5$ at the point $[f]$ is
the span of the tangent spaces of any $9$  points in $W_3$ whose
span contains $[f]$.  The tangent space to the $8$-th secant
variety at $[f]$ is therefore defined by the linear space of cubic
hypersurfaces that  are singular at $p_1,...,p_9$.   The curve $E$
is contained in four Veronese surfaces, corresponding to the four
square roots of the hyperplane line bundle of degree $6$. The secant
varieties of these Veronese surfaces generate a pencil of cubic
hypersurfaces singular  along the elliptic curve.  Their
intersection is precisely the union of secant lines to $E$, so there
are no other cubics singular along $E$, and $E$ is the common
singular locus of this pencil. We will show that these hypersurfaces
are precisely the cubic hypersurfaces singular at $p_1,...,p_9$.
Since the divisor $2D$ is not linearly equivalent
to a cubic hypersurface divisor, a cubic hypersurface singular in
$D$ must contain the curve $E$. Furthermore, on any
smooth intersection $S$ of three quadrics containing the curve, the curve $E$ has
trivial normal bundle.  Therefore, the residual of a cubic
hypersurface section  of $S$ that contains $E$,  meets the curve in a
divisor equivalent to a cubic section.  Hence, a cubic that is
singular along $D$, must contain the doubling of the
curve in the three quadrics.  Varying the complete intersection
surface $S$, we may conclude that the cubic must be singular along the
curve.

Summing up we see that tangent space of the $8$-th secant variety of
$W_3$ at the point $[f]$ has codimension $2$ in $\Pn {55}$ and that any $9$-tuple
of points on $W_3$ whose span contains $[f]$ is contained in the reembedding $E"$ of $E$ in $W_3$.

The curve $E"$ in $W_3$ is an elliptic normal curve of degree $18$.
By \cite[Proposition 5.2]{CC}, and its proof, when $2D$ is not equivalent to $3H_6$ there is unique divisor $D''$ of degree $9$ on $E"$,  distinct from $D$, whose span in the $3$-uple embedding contains  $[f]$.  In fact, $D''$ is a divisor equivalent to $3H_6-D$.  Therefore  $f$ is apolar to exactly two subschemes of length $9$ supported on $E$.
\end{proof}

\begin{lemma}{\label{ellip}} Let $F$ be a fourfold defined by a cubic form $f$, that has no partials of rank $\leq 3$ and
whose apolar ideal $I_f$ has a minimal free resolution with Betti numbers
 \[ \begin{array}{ccccccc}
 1 & - & - & - & - & -& - \\
 - & 15 & 35 & 21& - & -& - \\
 - & - & - & 21& 35 & 15& - \\
 - & - & - & - & - & -& 1
 \end{array}.
 \]
 Assume that the $(35\times 21)$-matrix $M_2$ of linear second order
syzygies has rank $21$ at a general point, and that there is a $10$-tuple of points $Z\subset \PP(V)$ apolar to $f$ and a point $[l]\in Z$ at which $M_2$ has rank $20$, such that the $9$ points $Z_0=Z\setminus [l]$ form a divisor $D$ on an elliptic sextic curve $E_{Z_0}$.  Assume furthermore that the divisor $2D$ is not equivalent to a cubic hypersurface divisor on $E_{Z_0}$,  and that $[l]$ is not contained in the secant variety of any Veronese surface containing $E_{Z_0}$.

Then the projection $I_{VSP}\to V_{10}(F)$ is generically $2:1$.

\end{lemma}
\begin{proof}
Note that the cubic form $f$ has rank $10$ and that $f-cl^3$ has rank $9$ for some $c\in \CC$.  In fact, since $M_2$ has rank $20$ at $[l]$, the ideal $I_f$ has a unique second order linear syzygy vanishing at $[l]$.  By Lemma \ref{second order lin}, this syzygy determines uniquely the curve $E_{Z_0}$. In particular, there is a unique $c$ such that $f-cl^3$ has rank $9$.  By Lemma \ref{twonines}, there are exactly two points in the fiber of the projection
\[
I_{VSP}\to V_{10}(F)\qquad  ([l],[Z])\mapsto [l]
\]
over $[l]$.  Since the conditions on $Z$ are open, the lemma follows.
\end{proof}

\section{The divisor of cubic fourfolds apolar to a Veronese \label{sec3}}

In the first part of this section we show (Corollary \ref{lesingver}) that for a general cubic fourfold $F$
apolar to a Veronese surface $\Sigma$, i.e. in the set $D_{V-ap}$,
the variety $VSP(F,10)$ is singular along a $K3$ surface, and then (Corollary \ref{lesingvdix}) that the hypersurface $V_{10}(F)$ introduced and studied in the previous
section is singular along $\Sigma$.
Subsequently, we show (Corollary \ref{lefiniteveronese}) that
the general $F$ in $D_{V-ap}$ is apolar to finitely many Veronese
surfaces, by exhibiting an $F$ in $D_{V-ap}$ such that the singular
locus of $V_{10}(F)$ cannot contain a one-dimensional family of
Veronese surfaces.
Next, we extend results in Section \ref{3.2} to
show (Corollary \ref{codim1} and Propositions \ref{propapdivapolar} and \ref{newprop19aout}) that $D_{V-ap}$ is a divisor different from $D_{rk3}$ and
$D_{copl}$, and that the fourfold $VSP(F,10)$ does not meet the singular locus of the Hilbert scheme for a general $F$ in $D_{V-ap}$ (Corollary \ref{corleb}).
In the final part of the section we show (Proposition
\ref{propnonNL}) that $D_{V-ap}$ is not a Noether Lefschetz divisor
in the moduli space of smooth cubic fourfolds.

The Propositions  \ref{propapdivapolar} and \ref{propnonNL} and Corollaries \ref{lefiniteveronese} and \ref{corleb} are applied in Section \ref{secpropsing} to show that for a general $[f]\in D_{V-ap}$, the fourfold $VSP(F,10)$ is smooth outside a surface along which it has quadratic singularities.

By a direct calculation in an example we now prove:
\begin{proposition} \label{exampleveronese} Let $F$ be a general
cubic fourfold apolar to a Veronese surface $\Sigma$.

 (i)  $F$ has cactus rank
$10$. Hence  no length $9$ subscheme of $\PP(V)$ is apolar to $F$.

(ii) $F$ is nonsingular and the form $f$ defining $F$ has no partial derivatives of rank
$\leq 3$.

(iii) The minimal free resolution of the apolar ideal $I_f$ has Betti numbers
\[ \begin{array}{ccccccc}
 1 & - & - & - & - & -& - \\
 - & 15 & 35 & 21& - & -& - \\
 - & - & - & 21& 35 & 15& - \\
 - & - & - & - & - & -& 1
 \end{array},
\]
and the matrix $M_2$ of linear second order syzygies of $I_f$ has rank $20$ at a general point of $\Sigma$.

(iv) $Y_F=V_{10}(F)$ is an irreducible fourfold singular in codimension at least $2$.
\end{proposition}
\begin{proof}

We find with {\it Macaulay2} \cite {MAC2} a cubic form apolar to a Veronese surface $\Sigma$, and compute the
 resolution of its annihilator (apolar ideal).
Let $\Sigma$ be the Veronese surface defined by the $(2\times 2)$-minors of
\[
\begin{pmatrix}x_0&x_1&x_2\\x_1&x_3&x_4\\x_2&x_4&x_5
\end{pmatrix}
\]
So the ideal of $\Sigma$ is generated by
\[
\langle x_0x_3-x_1^2, x_0x_5-x_2^2, x_3x_5-x_4^2,x_0x_4-x_1x_2,x_1x_4-x_2x_3,x_1x_5-x_2x_4.\rangle
\]
By differentiation one may check that each of these quadratic forms annihilates the following cubic form:

\begin{align*}
f=&y_0^2y_1+y_1y_2^2-2y_1y_2y_3-y_2y_3^2-y_1^2y_4+2y_0y_2y_4-2y_0y_3y_4-2y_1y_3y_4\\
&+2y_0y_1y_5+y_4^2y_5+y_3y_5
       ^2).
\end{align*}

So $f$ is apolar to the Veronese surface $\Sigma$.
The apolar ideal of the cubic form $f$ has Betti numbers
\[ \begin{array}{ccccccc}
 1 & - & - & - & - & -& - \\
 - & 15 & 35 & 21& - & -& - \\
 - & - & - & 21& 35 & 15& - \\
 - & - & - & - & - & -& 1
 \end{array}.
\]
Its $35\times 21$-matrix $M_2$ of second order linear syzygies restricted to the plane 

\[
x_0= x_3=x_4-x_5=0,
       \]
has rank $20$ along a curve of degree $9$.  Reduced modulo $5$ the defining form for this curve is

           \begin{align*}
&x_1^9-2x_1^8x_2+2x_1^7x_2^2-x_1^6x_2^3+x_1^5x_2^4+x_1^4x_2^5+2x_1^3x_2^6-2x_1^2x_2^7+2x_2^9-2x_1^8x_4\\
&+x_1^7x_2x_4-2x_1^6x_2^2x_4+2x_1^5x_2^3x_4-x_1^4x_2^4x_4-x_1^3x_2^5x_4-x_1^2x_2^6x_4+x_1x_2^7x_4\\
      & -2x_2^8x_4+2x_1^7x_4^2+2x_1^6x_2x_4^2-2x_1^5x_2^2x_4^2-x_1^4x_2^3x_4^2-2x_1^2x_2^5x_4^2-x_1x_2^6x_4^2-x_1^6x_4^3\\
      & -2x_1^4x_2^2x_4^3-2x_1^3x_2^3x_4^3+2x_1x_2^5x_4^3-x_2^6x_4^3+x_1^5x_4^4-2x_1^4x_2x_4^4-2x_1^3x_2^2x_4^4\\
      & -x_1x_2^4x_4^4+2x_1^4x_4^5-x_1^3x_2x_4^5-2x_1x_2^3x_4^5+x_2^4x_4^5-2x_1^3x_4^6+2x_1^2x_2x_4^6-2x_2^3x_4^6\\
      & +2x_1^2x_4^7+2x_1x_2x_4^7+x_1x_4^8-2x_2x_4^8-x_4^9
            \end{align*}

It is nonsingular, which proves (iv). In particular the generic rank of the matrix
$M_2$ is $21$ for $f$. Therefore, by Lemma \ref{rank9}, the cactus rank of $f$ is $10$, which proves (i).

A direct computation shows that $F=V(f)$ is nonsingular, that $f$  has no partials of rank $3$,  and that the matrix $M_2$ for the apolar ideal of
$f$ has rank $20$ at the point $V(x_0,x_1,x_2,x_3,x_4)\in \Sigma$, hence at a general point on $ \Sigma$, which proves (ii) and (iii), respectively.
\end{proof}

Let $W$ be a vector space of rank $3$, and $V=S^2W$, and recall the linear map (cf. (\ref{letrivial}))
\begin{eqnarray}\label{lineartrivial}
s:S^6W\rightarrow S^3V\qquad \text{s.t.}\;\;s(a^6)=(a^2)^3
\end{eqnarray}

For $g\in
S^6W$, we consider the cubic form $f=s(g)\in S^3V$.  Let $C=V(g)\subset \PP(W^*)$ and $F=V(f)\subset \PP(V^*)$.  Formula
(\ref{lineartrivial}) shows that there is a natural embedding
\[
\phi: VSP(C,10)\to VSP(F,10).
\]

Indeed, if
$g=\sum_ia_i^6$, then $f=s(g)=\sum_i(a_i^2)^3$. For distinct $[a_i]\in
\PP(W)$, the morphism $\phi$ sends the length $10$ subscheme
$\{[a_i]|i=1,...,10\}$ to the length $10$ subscheme $\{[a_i^2]|i=1,...,10\}$ of $\PP(V)$. More
generally, $\phi$ associates to a length $10 $ apolar subscheme $Z$
of $g$ in $\PP(W)$ the length $10 $ apolar subscheme to $f$ in
$\PP(V)$ which is the image of $Z$ under the Veronese embedding.
\begin{remark}\label{Sg}{\rm
When $g$ is
general sextic ternary form and $C=V(g)$, then $g$ has rank $10$.  Mukai showed in \cite{Muk} that  $VSP(C,10)$ is a smooth $K3$ surface.
We shall often use the notation $S_g:=\phi(VSP(C,10))$.}
\end{remark}

\begin{lemma}\label{dim4} If $g$ is a general sextic ternary form and $f=s(g)$, then $VSP(F,10)$ is a fourfold and the projection $I_{VSP}\to V_{10}(F)$ is generically $2:1$, where $F=V(f)$.
\end{lemma}
\begin{proof} We may assume that $f=s(g)$ is a general cubic form apolar to a Veronese surface.   By Remark \ref{Sg}, the sextic form $g$ has rank $10$, and  by  Proposition \ref{exampleveronese}, $f=s(g)$ has cactus rank $10$, hence also rank $10$, and $V_{10}(F)$ is a fourfold.  We first claim that the projection $I_{VSP}\to V_{10}(F)$ is generically finite, and hence that $VSP(F,10)$ is also a fourfold.  For this let $\Sigma$ be the Veronese surface that is the image of the quadratic embedding of $\PP(W)$ in $\PP(V)=\PP(S^2W)$, and let $Z\subset \Sigma$ be a general $10$-tuple of points on $\Sigma$ that is apolar to $F$.  Note that $[Z]\in S_g=\phi(VSP(C,10))$.  We may assume that a subset  $Z_0\subset Z$  of $9$ points lie on a unique elliptic sextic curve $E_{Z_0}$ in $\Sigma$.  Then $p=Z\setminus Z_0$ lies on $V_{10}(F)$, and $I_\Sigma$ has a unique linear second order syzygy that vanishes at $p$. By Proposition \ref{exampleveronese} (iii), this syzygy is the unique second order syzygy for $I_f$ that vanishes at $p$, and, by Lemma \ref{second order lin}, every apolar subscheme $Z'$ to $F$ that contains $p$ must be contained in $\Sigma$.  Thus $[Z']\in\phi(VSP(C,10))$.  But, by Remark \ref{Sg},
 the variety $VSP(C,10)$ is a surface, so, since $p$ is a general point,  there are finitely many apolar schemes to $C$ of length $10$  that contain the point $p$.  Therefore the projection $I_{VSP}\to V_{10}(F)$ also has a finite fiber over $p\in V_{10}(F)$, and the projection is generically finite.

 By Proposition \ref{exampleveronese} (iv), the variety $V_{10}(F)$ is an integral hypersurface of degree $9$.
 Consider a one parameter family of cubic fourfolds $\{F_t\}_{t\in \CC}$ that contains $F$ and the total family
 $$I=\{(p,Z,t)|p\in Z, [Z]\in VSP(F_t,10)\}\subset \PP(V)\times {\rm Hilb}_{10}\PP(V)\times \CC.$$
 After possibly shrinking the parameterspace, we may assume that $I$ is irreducible and flat over an open subset $\Delta\subset\CC$.  For $t\in \Delta$ the fiber in $I$ over $t$ is $I_{VSP_t}$, and the variety  $V_{10}(F_t)\subset \PP(V)$ is the image of the projection of this fiber into the first factor. For general $t$ the variety $V_{10}(F_t)$ is an integral hypersurface of degree $9$ by Lemma \ref{leautiliser}, while the  projection $I_{VSP_t}\to V_{10}(F_t)$ is generically $2:1$ by Lemma \ref{ellip}.  Since $V_{10}(F)$ is  a hypersurface of degree $9$, the generically finite map  $I_{VSP}\to V_{10}(F)$ is also $2:1$.

\end{proof}

To show that $VSP(F,10)$ is singular along $S_g=\phi(VSP(C,10))$, we use the following general criterion for singularities of the variety of power sums of a hypersurface:

\begin{lemma}\label{singvsp} Assume that $k$ is the rank of a general hypersurface $F'$ of degree $d$ in $\PP(V^*)$.
Let $F\subset \PP(V^*)$ be a hypersurface of degree $d$ and rank $k$ and assume that $\dim VSP(F,k)=\dim VSP(F',k)$. Let $[Z]\in\text{Hilb}_k({\PP(V)})$ be an apolar subscheme to $F$ such that
$Z=\{l_1,\ldots,l_k\}$ consists of $k$ distinct points.
Then
$VSP(F,k)$ is singular at  $[Z]$,
 if
there is a hypersurface of degree $d$ in ${\PP(V)}$ which is
singular along $Z$.
\end{lemma}
\begin{proof}
Consider the universal family
 $$\mathcal{V}SP=\{([Z],[f])|[Z]\in VSP(F,k)\}\subset \text{Hilb}_k({\PP(V)})\times {\bf
 P}(S^dV).$$
  The
 fiber of the second
 projection over  a point  $[f]\in {\bf P}(S^dV)$ is $VSP(F,k)$, where $F=V(f)$. The
 fiber of the first projection over a point
 $[Z]\in \text{Hilb}_k(\PP(V))$  is a linear space, the linear
 span $\langle \rho_d(Z)\rangle$ of the image
 $\rho_d(Z)$ in ${\bf P}(S^dV)$ under the $d-$uple Veronese embedding
 $\rho_d$. Now, consider a point $([Z],[f])\in \mathcal{V}SP$ where $Z$ is a smooth subscheme apolar to $f$ and $f$ has rank $k$.  Then $Z$ belongs to
 the set of subschemes that impose independent
conditions to polynomials of degree $d$, which is open in the Hilbert scheme, and $\langle \rho_d(Z)\rangle$ is  a $\Pn {k-1}$.
 Since $\text{Hilb}_k(\PP(V))$ is smooth of  dimension $kn$ near $Z$, we conclude that
  $\mathcal{V}SP$ is smooth of
 dimension $kn+k-1$ at $([Z],[f])$.  Since $F$ has rank $k$,
the second projection  $\mathcal{V}SP\to {\bf P}(S^dV)$ is
dominant.  Furthermore, since the dimension of the fiber $VSP(F,k)$ of the second projection $\mathcal{V}SP\to {\bf P}(S^dV)$ is equal to the dimension of a general fiber,  the variety $VSP(F,k)$ is singular at a point $[Z]$ if
the rank of the second projection
at the point $([Z],[f])$ is less than ${\rm dim}( {\bf
P}({S}^dV))$.   If $Z=\{[l_1], ...,[l_k]\}$, then this rank is the
dimension of the span
 $T_Z=\langle [l_i^{d-1}y_j]|1\leq i\leq k, 0\leq
 j\leq n\rangle$ where $\langle y_0,\ldots,y_n\rangle=V$. In fact, from the expansion of
 $(l_i+y_j)^d$, we see that $l_i^{d-1}y_j$ defines a tangent direction at the point
 $[l_i^d]$, so $T_Z$ is the span of the tangent spaces to the $d$-uple
 embedding $\rho_d({\bf P}(V))$ at the points $[l_i^d]$ (this is a special case of Terracini's Lemma,
 cf. \cite{Ter}, \cite{Zak}). Hence $VSP(F,k)$ is singular at $[Z]$
 if these tangent spaces do not span ${\bf
P}({S}^dV)$.

But hyperplanes in ${\bf P}(S^dV)$ correspond to hypersurfaces of degree $d$ in $\PP(V)$, and a hyperplane contains the tangent space at $[l_i^d]$ if  and only if the corresponding hypersurface is singular at $[l_i]$.
Therefore $VSP(F,k)$ is singular at $[Z]$
  if  there is a hypersurface in $\PP(V)$ of degree $d$ singular in the points
 $[l_1],\ldots, [l_{k}]$.
\end{proof}
\begin{corollary}\label{lesingver} If $g$ is a general sextic ternary form and $f=s(g)$ (cf. \ref{lineartrivial}), then $VSP(F,10)$ is
singular along $S_g=\phi(VSP(C,10)$, where $F=V(f)$ and $C=V(g)$.
 \end{corollary}
 \begin{proof}  First, we may assume that $f=s(g)$ is a general cubic form apolar to a Veronese surface $\Sigma$.   By Remark \ref{Sg}, the sextic form $g$ has rank $10$, and  by  Proposition \ref{exampleveronese}, $f=s(g)$ has cactus rank $10$, hence also rank $10$, and  by Lemma \ref{dim4}, 
 $VSP(F,10)$ is a fourfold.

 Let $Z\subset \Sigma$ be a general $10$-tuple of points on $\Sigma$ that is apolar to $F$.  Note that $[Z]\in S_g=\phi(VSP(C,10))$.
 Now, since $g$ has rank $10$, the points in $Z$ impose independent
 conditions to $S^3V^*$.  According to Lemma
 \ref{singvsp}, the variety
 $VSP(F,10)$ is singular at $[Z]$ if  there exists a cubic fourfold
 singular along $Z$.
 This condition is satisfied in our situation since $Z$ is contained
 in the Veronese surface $\Sigma\subset \PP(V)$.  In fact, the Veronese surface is the singular locus of  the discriminant cubic hypersurface
 parameterizing singular conics in $\PP(W^*)$.
\end{proof}
The next lemma is used to prove that if $F$ is a general cubic fourfold apolar to a Veronese surface $\Sigma$, then the hypersurface $V_{10}(F)$ is singular along $\Sigma$
(Corollary \ref{lesingvdix}).
 \begin{lemma}{\label{ellipres}} Let $C$ be a plane curve defined by a general sextic form $g$, and let
 \[
 I_{VSP}=\{([l],[Z])|[l]\in Z\}\subset \PP^2\times VSP(C,10)\
 \]
  be the natural incidence variety.
 Then the projection onto the first factor is $2:1$.
\end{lemma}
\begin{proof} We may assume that $g$ has rank $10$.  Let $p=[l]\in Z\subset \PP^2$ be a point
in a general apolar subscheme of length $10$.  Then $Z_0=Z-p$ has length nine and is contained in a unique smooth cubic curve $E_{Z_0}$. In fact, in the pencil $g-\lambda l^6$, there is a unique form $g_1$ of rank $9$:  The $(10\times 10)$-catalecticant matrix of $\mu g-\lambda l^6$ has nonvanishing determinant with a zero of multiplicity $9$ at $\mu=0$, and hence one more zero.  The corank of the catalecticant matrix is the rank of the space of cubic forms apolar to $\mu g-\lambda l^6$, so the simple zero correspond to a unique sextic form $g_1$ in the pencil $g-\lambda l^6$ that is apolar to a cubic curve, i.e. apolar to the scheme $Z_0$ and the cubic curve $E_{Z_0}$.
By genericity, we may assume that  $2Z_0$ is not equivalent to $6H_L$ as divisors on $E_{Z_0}$, where $H_L$ is a divisor defined by a line in  $\PP^2$.
We apply now Lemma \ref{twonines} to the cubic form $s(g_1)$, and conclude that $p$ is contained in exactly two subschemes of length $10$ that are apolar to $g_1$.
\end{proof}

\begin{corollary}\label{lesingvdix}  Let $F$ be a cubic fourfold of cactus rank $10$ that is apolar to a Veronese surface $\Sigma$.
If the matrix $M_2$ of linear second order syzygies of the apolar ideal $I_f$ has rank $20$ at a general point of $V_{10}(F)$, then this hypersurface is singular along $\Sigma$.
\end{corollary}
\begin{proof}
 Let $F=V(f)$ and $f=s(g)$, then, by Corollary \ref{lesingver},
the variety $VSP(F,10)$ is  singular along the $K3$ surface
$S_g=\phi(VSP(C,10))$, where $C=V(g)$, and $V_{10}(F)\subset \PP(V)$ is a hypersurface of degree $9$, by Lemma \ref{leautiliser}.
Now, assume that $\{g_t\}_{t\in \CC}$ is a general one parameter family of ternary sextic forms such that $g=g_0$.  Let $g_t$ be a general member of the family.
Then, by Proposition \ref{exampleveronese} and Remark \ref{Sg}, the sextic form $g_t$ and the cubic form $f_t=s(g_t)$ both have rank $10$.   Any length $9$ subscheme of a general  apolar scheme of length $10$ of $g_t$ is contained in an elliptic normal sextic curve $E$ on $\Sigma$, and as a divisor on $E$ satisfies the condition of Lemma \ref{twonines}.
The projection $I_{VSP_t}\to VSP(F_t,10)$ and its restriction over $S_{g_t}$ are both finite and of
degree $10$.

Consider the other projection $I_{VSP_t}\to V_{10}(F_t)$. By Lemma \ref{dim4}, it is generically $2:1$ and  $V_{10}(F_t)$ is an irreducible hypersurface of degree $9$.
 On the other hand, by Proposition \ref{exampleveronese} {(iii)} the Veronese surface $\Sigma$ is contained in $V_{10}(F_t)$.  Let $p\in \Sigma$ be a general point.  As in the proof of Corollary \ref{lesingver}, if $(p,[Z])\in I_{VSP_t}$, then $Z\subset \Sigma$.  Therefore we may conclude, by  Lemma %s \ref{ellip} and
\ref{ellipres}, that the projection $I_{VSP_t}\to V_{10}(F_t)$ is  $2:1$ over $p$, and hence generically over $\Sigma$.

 An analytic neighborhood in
$VSP(F_t,10)$ of a general point in $S_{g_t}$ is therefore isomorphic to a
suitable neighborhood in $V_{10}(F_t)$ of any of the corresponding
points in $\Sigma$. Therefore $V_{10}(F_t)$ is singular along $\Sigma$
if and only if $VSP(F_t,10)$ is singular along $S_{g_t}$.

Thus, for an open neighborhood $0\in\Delta\subset \CC$, there is a family $\{V_{10}(F_t)|t\in \Delta\}$ of hypersurfaces of degree $9$  whose general member is singular along $\Sigma$.
To see that $V_{10}(F)$ is singular along $\Sigma$, we study this family.

Let
 $$I=\overline{\{(p,Z,t)|p\in Z, [Z]\in VSP(F_t,10)\}}\subset \PP(V)\times {\rm Hilb}_{10}\PP(V)\times \Delta,$$
 be the closure of the natural incidence, and let ${\mathcal V}SP_\Delta(10)\subset \PP(V)\times \Delta$ be the image of the projection $I\to \PP(V)\times \Delta$.
 Then $I$ and ${\mathcal V}SP_\Delta(10)$ are fivefolds, and thus the fibers of ${\mathcal V}SP_\Delta(10)\to \Delta$ are all fourfolds.  The fiber at $t\in\Delta$ therefore contains the fourfold $V_{10}(F_t)$ as a component.  Since  $V_{10}(F_t)$ is singular along $\Sigma$ for a general $t$, the same holds for  $t=0$ and the lemma follows.
\end{proof}
\begin{remark}
{\rm In computations we have found forms $f$ apolar to a Veronese
surface $\Sigma$, such that $V_{10}(F)$ is singular along the union
of $\Sigma$ and a surface of degree $140$, the locus of points where
the matrix $M_2$ of second order linear syzygies has rank at most
$19$. As noted in Proposition \ref{exampleveronese} (iii), the matrix $M_2$ has rank $20$ generically on $\Sigma$.}
\end{remark}

\begin{corollary} \label{lefiniteveronese} Let $F$ be a general
cubic fourfold apolar to a Veronese surface.  Then  $F$  is  apolar to finitely many Veronese surfaces.
\end{corollary}
\begin{proof}
The union of a $1-$dimensional family of Veronese surfaces is a threefold.
 So, if $f$ is apolar to a $1-$dimensional family of Veronese surfaces, then, by Lemma \ref{leautiliser} and Corollary  \ref{lesingvdix},
 the degree $9$ determinantal hypersurface $V_{10}(f)$ would
 be singular along a threefold, contradicting Proposition \ref{exampleveronese}. 
 \end{proof}

\begin{corollary}  \label{codim1} The set $D_{V-ap}$ of cubic forms that are apolar to some Veronese surface is an irreducible hypersurface
in $\PP(S^3V)$.
\end{corollary}
\begin{proof}  The map $g\mapsto f=s(g)$ induces
 a rational map
 $$s_{mod}:S^6W//Gl(W)\dashrightarrow S^3V//Gl(V).$$
The image of $s_{mod}$ is the locus of cubic fourfolds apolar
 to a Veronese surface, so $D_{V-ap}$ is irreducible.  That it is a divisor,  follows from a dimension count:  Plane sextics have $28-9=19$ parameters, while cubic fourfolds have $56-36=20$ parameters, so
  it suffices to
 show that $s_{mod}$ has generically finite fiber.  The fiber
 of $s_{mod}$ over a point parameterizing a cubic fourfold $F$ may be identified with the set of Veronese surfaces
 which are apolar to $F$. The result thus follows from
 Corollary \ref{lefiniteveronese}.
\end{proof}

\begin{proposition}\label{propapdivapolar} A general cubic fourfold $F$ which is apolar to a Veronese
surface satisfies the conclusion of Proposition \ref{allapolar},
namely, any element $[Z]\in VSP(F,10)$ corresponds to a length $10$
 subscheme $Z$  which imposes independent conditions on cubics and is apolar to $F$.
\end{proposition}
\begin{proof} By Proposition \ref{allapolar}, the divisorial part of the
set of cubic fourfolds not satisfying this conclusion is contained
in the union of  the irreducible divisors $D_{rk3}$ and $D_{copl}$
introduced in Section \ref{sec1}. As we know that the set of cubics
apolar to a Veronese surface is an irreducible divisor which is
different from $D_{rk3}$ by Proposition \ref{exampleveronese} (ii), the result
follows from the following Proposition \ref{newprop19aout}.
\end{proof}

\begin{corollary} \label{corleb} For a general cubic fourfold $F$ which is apolar to a Veronese surface,
$VSP(F,10)$ does not meet $\text{Sing}(\text{Hilb}_{10}(\PP(V))$.
\end{corollary}
\begin{proof} This follows from
Proposition \ref{exampleveronese}(ii) which guarantees that the form $f$ that defines $F$ has no
partial derivative of rank
 $\leq3$, Proposition \ref{propapdivapolar}
 and Proposition \ref{corlea}.
\end{proof}

\begin{proposition}\label{newprop19aout} The divisors
$D_{copl}$ and $D_{V-ap}$ are distinct.
\end{proposition}
\begin{proof}
 We shall distinguish $D_{copl}$ and $D_{V-ap}$ by proving that  their intersections
$$D_{V-ap}\cap D_{IR}\quad {\rm and} \quad D_{copl}\cap D_{IR}$$ with $D_{IR}$ are
distinct.

Recall from Section \ref{sec1} that $D_{IR}$ denotes the set of cubic fourfolds $F_{IR}(S)$ associated to a $K3$
 surface section $S=\PP_S^8\cap G(2,6)\subset \PP^{14}$.   The dual space $\PP_S^5:=(\PP_S^8)^{\bot}\subset \check{\PP^{14}}$ intersects the Pfaffian cubic hypersurface, the secant variety of $\check{G}(2,6)\subset \check{\PP^{14}}$, in a Pfaffian cubic fourfold $F_{BD}(S)$.  Furthermore, by \cite[Lemma 3.9 and Proposition 3.15]{IR}, there is an identification
 $V_{10}(F_{IR}(S))=F_{BD}(S)$.

\begin{lemma} \label{BDS} Let $F_{BD}(S)\subset \PP_S^5$ be a Pfaffian cubic fourfold  with no rank $2$ points, i.e. $ \PP_S^5\cap \check G(2,6)=\emptyset$, and let  $S=\PP_S^8\cap G(2,6)$ be the corresponding linear section of $G(2,6)$.   Then $F_{IR}(S)$ has cactus rank $10$, and $S$ is birational to a component of the Hilbert scheme of rational quartic surface scrolls in $F_{BD}(S)$ that are apolar to $F_{IR}(S)$.
\end{lemma}
 \begin{proof}  If $F_{BD}(S)$ has no rank $2$ points, then $\PP_S^8=(\PP_S^5)^{\bot}$ defines the apolar ideal $I_f$ of a cubic fourfold $F_{IR}(S)=V(f)$, and this fourfold has cactus rank $10$, (cf. \cite[ 3.5 and Lemma 3.6]{IR}).  By \cite[Lemma 2.9]{IR}, each secant line to $S$ defines a pair of rational quartic surface scroll in $F_{BD}(S)$ that intersect along scheme of length $10$ apolar to  $F_{IR}$.  The two scrolls correspond to the points of intersection on the variety $S$.
  \end{proof}

If a cubic fourfold $F$ of cactus rank $10$ is apolar to a Veronese surface, then, by Corollary \ref{lesingvdix},  $V_{10}(F)$ must contain this Veronese surface.
 So the proposition follows by finding a cubic fourfold $F=F_{IR}(S)\in D_{copl}\cap D_{IR}$, such that the Pfaffian cubic $F_{BD}(S)$ contains no Veronese surface.

We first consider Pfaffian cubic fourfolds that contain a plane.
  For a smooth cubic fourfold $F$, let $A(F)=H^4(F,\ZZ)\cap H^{2,2}(F)$, the lattice of integral middle Hodge classes.
\begin{lemma}\label{A(F)} If $F$ is a general smooth Pfaffian cubic fourfold that contains a plane $P$  intersecting a rational quartic surface $S_4$ in $F$ along a conic section, then $A(F)$ does not contain the class of a Veronese surface.  Furthermore the Pfaffian cubic fourfolds that contain such a plane form a divisor in the variety  of Pfaffian cubic fourfolds.
\end{lemma}
\begin{proof}  We may assume that $A(F)$ has rank $3$, generated by the classes of $h^2, [S_4]$ and $[P]$ (cf. \cite[Example 3.1 and Theorem 3.] {BR}).   The intersection matrix is
\[
\begin{matrix}
          &          h^2 &     [S_4]           &      [P]           \\
                h^2 & 3 &     4            &    1       \\
            [S_4] &    4 & 10 &     0           \\
            [P] &    1 &     0            &  3
        \end{matrix}
        \]
       The class of a  Veronese surface $\Sigma$ in $F$ would have intersections
       \[
       h^2\cdot [\Sigma]=4, [\Sigma]^2=12
       \]
    and $ -1\leq [\Sigma]\cdot P \leq 3$, since $\Sigma$ has degree $4$ and $[\Sigma]\cdot P=-1$ when $\Sigma\cap P$ is a conic.   If $\Sigma\subset F$, then  $[\Sigma]=aH^2+b[S_4]+c[P]$ with integral $a,b,c$.
    Computing intersection numbers we get
    $$3[\Sigma]^2  - ([\Sigma]\cdot [h^2])^2=14b^2+8c^2-8bc=20.$$
    Since $4$ divides the right hand side,  $b$ in the left hand side must be even, which again means that $8$ must divide the right hand side, a contradiction.
    For the last statement, see  \cite[Section 2]{ABBVA}. \end{proof}

  We now exhibit a cubic form $f$ in $D_{copl}$  that is apolar to a rational quartic surface scroll and has both rank and cactus rank
  $10$. By Lemma \ref{DIR}, it belongs to $D_{copl}\cap D_{IR}$.
 
The cubic form

\begin{align*}
f&= -7x_0^3+9x_0^2x_1-12x_0x_1^2+x_1^3-12x_0^2x_2+6x_0x_1x_2+3x_0^2x_3-3x_1^2x_3-6x_0x_2x_3\\&-6x_1x_2x_3+3x_0x_3^2
-3x_2x_3^2+x_3^3-6x_0x_1x_4-3x_1^2x_4-6x_0x_2x_4-6x_1x_3x_4-3x_0x_4^2\\&-3x_0^2x_5-6x_0x_1x_5+6x_2^2x_5-6x_0x_3x_5
+6x_2x_4x_5+3x_1x_5^2-x_5^3
\end{align*}

is apolar to the two quartic surface scrolls $S_4$ and $S_4'$ defined by the $2$-minors of
  \[
   \begin{pmatrix}
     x_0 &x_1&x_3&x_4\\
     x_1&x_2&x_4&x_5
\end{pmatrix}\; {\rm and}\;
\begin{pmatrix}
     x_3 &x_4&x_0+x_1+x_5&x_1-x_2+x_4\\
     x_4&x_5&x_1-x_3+x_4&x_0+x_2-x_3
\end{pmatrix}
\]
respectively,  so $f$ belongs to $D_{IR}$.
The intersection $S_4\cap S_4'$ of the two scrolls is the union
of the six points $Z_6$ defined by the $2$-minors of
    \[
    \begin{pmatrix}
     x_0 &x_1&x_3&x_4&x_0+x_1+x_5&x_1-x_2+x_4\\
     x_1&x_2&x_4&x_5&x_1-x_3+x_4&x_0+x_2-x_3
\end{pmatrix}
\]
and the four points
   \[
   V(x_0x_2-x_1^2,x_0^2+x_0x_1+2x_1x_2), x_3,x_4,x_5)
   \]
   in the plane $V(x_3,x_4,x_5)$,   so $f$ belongs also  to $D_{copl}$ and has rank at most $10$.
   The resolution of the  apolar ideal $I_f$ has Betti numbers
       \[ \begin{array}{ccccccc}
 1 & - & - & - & - & -& - \\
 - & 15 & 35 & 21& - & -& - \\
 - & - & - & 21& 35 & 15& - \\
 - & - & - & - & - & -& 1
 \end{array}.
 \]
and the matrix ${35\times 21}$-matrix $M_2$ of second order linear
syzygies of $I_f$ 
 has no
syzygies.  So we conclude that $f$ has cactus rank $10$ by Lemma
\ref{rank9} and hence also rank $10$.
Let $F=V(f)$.  Then $F=F_{IR}(S)$ for some $K3$ surface $S$, and $F_{BD}(S)$ is the corresponding Pfaffian cubic that contains the two quartic scrolls $S_4$ and $S_4'$.  Since each scroll intersects the plane $P=V(x_3,x_4,x_5)$ in a conic section, the plane $P$ is contained in $F_{BD}(S)$ and $A(F_{BD}(S))$ contains the rank three lattice generated by $h^2, [S_4]$ and $[P]$ with intersection matrix as in the proof of Lemma \ref{A(F)}.   By Lemma \ref{A(F)}, the Pfaffian cubic fourfolds $F$ that contain a plane that intersect a quartic scroll in a conic form a family of codimension one in the divisor of Pfaffian cubic fourfolds.
But  $D_{IR}\cap D_{copl}$ is a divisor in $D_{IR}$,  so the corresponding set of Pfaffian cubics also has codimension one in the divisor of Pfaffian cubic fourfolds.  Therefore the Pfaffians cubic fourfold $F_{BD}(S)$ corresponding to a general cubic fourfold  $F_{IR}(S)\in D_{IR}\cap D_{copl}$ is general in the sense of Lemma \ref{A(F)}, and does not contain the class of a  Veronese surface.

This concludes the proof of Proposition \ref{newprop19aout}.
\end{proof}

We conclude this section with the following result concerning the
divisor $D_{V-ap}$.
\begin{proposition} \label{propnonNL}
 The divisor $D_{V-ap}$  is not a Noether-Lefschetz divisor.
\end{proposition}

Here by a Noether-Lefschetz divisor (or component of the Hodge
loci, see  \cite{voisinhodgeloci}), we mean a divisor $D$ along
which a locally constant nonzero primitive rational cohomology class
in $H^4(F_b,\QQ),\,b\in D$, remains a Hodge class. Equivalently, as
the Hodge conjecture is satisfied by cubic fourfolds, the cubic
fourfolds $F_b$ parameterized by such a divisor carry a codimension
$2$ cycle whose cohomology class is not proportional to the class
$h^2$, $h=c_1(\mathcal{O}_{F_b}(1))$. Hodge theory shows that in the
case of cubic fourfolds, the Hodge loci are hypersurfaces in the
moduli space, as a consequence of the equality $h^{3,1}(F)=1$ (see
\cite{voisinhodgeloci}).
\begin{proof}[Proof of Proposition \ref{propnonNL}]  First of all, we recall
 that in the moduli stack $\mathcal{M}$ of smooth cubic fourfolds (or in the local universal family of deformations),
 Noether-Lefschetz divisors have a smooth normalization. More
 precisely, each local branch $\mathcal{M}_{\alpha}$ near a cubic fourfold $[F]$ is defined by a
 class $\alpha\in H^4(F,\QQ)_{prim}$, where
 $\mathcal{M}_\alpha$ is the ``locus of points $t\in\mathcal{M}$
 where the class $\alpha_t\in H^4(F_t,\QQ)_{prim}$  deduced from $\alpha$ by parallel transport is a Hodge class'',
 and the statement is that $\mathcal{M}_\alpha$ is smooth.
 We refer to \cite{voisinhodgeloci} for  various local descriptions
 of these Hodge loci and their local study.  The smoothness follows
 from \cite[Corollary 3.3]{voisinhodgeloci}, and from the following fact:
 \begin{lemma} Let $F$ be a nonsingular cubic fourfold, and $0\not=\alpha \in
 H^2(F,\Omega_F^2)_{prim}$.
  Then the cup-product-contraction map
  $$\lrcorner\alpha : H^1(F,T_F)\rightarrow H^3(F,\Omega_F)$$
  is surjective.
 \end{lemma}
 This lemma can be proved directly using Griffiths' description of
 the infinitesimal variations of Hodge structures of hypersurfaces,
 or by using the  Beauville-Donagi  isomorphism between the
 variation of Hodge structures on $H^4(F,\QQ)_{prim}$ and the
 variation of Hodge structures on $H^2(L(F),\QQ)_{prim}$,
 where $L(F)$ is the Fano variety of lines of $F$, together with general properties of the period map
 for hyper-K\"ahler manifolds.

 The  universal family of deformations of the cubic Fermat
 hypersurface  $F_{Fermat}=V(f_{Fermat})$ in $\PP^5$ can be obtained  as follows: in
 $S^3V$ we choose a linear subspace $T$  which is transverse to the
 tangent space  at the point
 $f_{Fermat}$ to the orbit of $f_{Fermat}$ under $Gl(V)$, and we
 restrict the universal hypersurface in $S^3V\times \PP^5$ to
 $T\times\PP^5$, where $T$ is embedded in an affine way in
$S^3V$, by $t\mapsto f_{Fermat}+t$. Since  the differential at $(Id,0)$ of the
map $$Gl(V)\times T\rightarrow S^3V,$$
$$(\gamma,t)\mapsto \gamma(f_{Fermat})+t,$$
is an isomorphism,  it is a local isomorphism in  the analytic
topology, hence there is a neighborhood $U'$ of $f_{Fermat}$ in
$S^3V$ and a holomorphic retraction $\pi:U'\rightarrow U\subset T$
with the property that $\pi(g)$ is the unique point of intersection
of $U'\cap O_g$ with $T$ (where $O_g$ is the orbit of $g\in U$ under
$Gl(V)$).

 It is well-known (see \cite[Remark 6.16]{voisinbook}) that
the tangent space to the orbit of $f_{Fermat}$ at
$f_{Fermat}$ is the degree $3$ part of the Jacobian ideal of $f_{Fermat}$,
generated by the partial derivatives of $f_{Fermat}$. If we write
$f_{Fermat}=\sum_{i=0}^{i=5}X_i^3$, the Jacobian ideal $J_{f_{Fermat}}$ is generated by the $X_i^2$, so
 there is a natural such complementary subspace $T$;  the vector subspace
of $S^3V$ generated by the $X_iX_jX_k$ for $i,\,j,\,k$ all distinct.

As the map $s_{mod}: S^6W//Gl(W)\dashrightarrow S^3V//Gl(V)$ is induced by the  linear map $s:S^6W\rightarrow S^3V$, the divisor
$D_{V-ap}\subset S^3V//Gl(V)$ comes from a divisor $D_U$ in $U\subset T\subset S^3V$ ($U$ is
a analytic open set which will be the basis of a universal family of
deformations of $F_{Fermat}$), where $D_U$ is obtained as the image of the composition of the linear map
$s:S^6W\rightarrow S^3V$ with $\pi:U'\rightarrow U\subset T$, where it is defined.

The following proposition implies that $D_{V-ap}$ is not a
Noether-Lefschetz divisor, thus concluding the proof of the
proposition.
\end{proof}
 \begin{proposition}\label{proDNL} The local branches of the
divisor $D_{U}$ at the origin are singular.
\end{proposition}
\begin{remark}{\rm We cannot identify here $D_U$ with an open set of  $D_{V-ap}$. Indeed, $D_U$ is a divisor in the
universal family of deformations of $F_{Fermat}$, and its image $D_{V-ap}$ in $S^3V//Gl(V)$
is obtained by taking the quotient of $D_U$ by the group of automorphisms of $F_{Fermat}$, which is nontrivial.
If $D_{V-ap}$ is a Noether-Lefschetz divisor, then the divisor $D_U$ in the universal family of deformations must have smooth local branches.
The criterion, that a Noether-Lefschetz divisor has smooth local branches can be applied only in the universal family of deformations, which is itself smooth.
}
\end{remark}
\begin{proof}[Proof of Proposition \ref{proDNL}] We wish to exploit the following observation:
\begin{lemma} For a generic  sextic polynomial  $g\in S^6W$ which is the
sum of six $6$-th powers of elements of $W$, $f=s(g)$ is (conjugate
to) the Fermat polynomial  $g_F=\sum_{i=0}^{i=5}X_i^3$.

\end{lemma}
\begin{proof} This follows immediately from formula (\ref{letrivial}), which
says that if $g=\sum_{1=0}^{i=5}a_i^6$ then $f=\sum_{i=0}^{i=5}
(a_i^2)^3$. On the other hand, for a generic choice of the $a_i$'s,
the $a_i^2$ provide a basis $X_i, i=0,...,5$ of $V=S^2W$.
 \end{proof}
We fix $a_0,\ldots ,a_5$ providing a basis $X_i=a_i^2, i=0,...,5$ of $V$. For
any ${b}_\bullet=(b_0,\ldots,b_5)\in W^6$ and $b\in W$, we consider
the curve in $S^6W$ parameterized by the coordinate $t$,  of the
form
$$t\mapsto g_{b_\bullet,b,t}:=\sum_{i=0}^{i=5}b_i^6+tb^6\in S^6W.$$
At $t=0$, the corresponding  curve $t\mapsto s(g_{b_\bullet,b,t})\in
S^3V$ passes through $s(\sum_{i=0}^{i=5}b_i^6)$, which is equal to
$\sum_{i=0}^{i=5}(b_i^2)^3\in S^3V$. The later polynomial is not
equal for generic $b_\bullet$ to the Fermat polynomial
$f_{Fermat}=\sum_iX_i^3$ but it is canonically conjugate to it,
namely, let $\gamma_{b_\bullet}\in Gl(V)$ be determined by
$$\gamma_{b_\bullet}(b_i^2)=X_i,\,i=0,\ldots, 5.$$
Then we have
$$\gamma_{b_\bullet}(s(\sum_{i=0}^{i=5}b_i^6))=f_{Fermat},$$
and may conclude that the curve
$$t\mapsto f_{b_\bullet,b,t}:=\gamma_{b_\bullet}(s(g_{b_\bullet,b,t}))\in S^3V,\,t\in \CC$$
passes through $f_{Fermat}$ at $t=0$. By definition, its image in
$S^3V//Gl(V)$ is contained in ${\rm Im}\,{s}_{mod}$. Furthermore,
for small $t$, $f_{b_\bullet,b,t}$ belongs to the small open set
where the holomorphic retraction $\pi:U\rightarrow T$ is  defined,
so that
 $\pi(f_{b_\bullet,b,t})\in D_{U}$ for any such
$(b_\bullet,t)$.  Thus there must be one branch $D'_U$ of $D_U$ such
that $\pi(f_{b_\bullet,b,t}))\in D'_U$ for any $(b_\bullet,t)$,
since the parameter space for the family $f_{b_\bullet,b,t}$ is
smooth hence in particular normal.
Let us now prove that $D'_U$ is not smooth at the point $f_{Fermat}$.
The
derivative at $0$ with respect to  $t$ of the  holomorphic map
$$t\mapsto \pi(f_{b_\bullet,t})
\in T$$ is obtained by applying the projection
$$p:S^3V\rightarrow T\cong S^3V/J_{f_{Fermat}}$$
 to
$\gamma_{b_\bullet}(s(b^6))=\gamma_{b_\bullet}((b^2)^3)$. The
above reasoning shows that all these elements lie in the Zariski
tangent space $T_{D'_U,0}$ at the point $0$ (parameterizing the Fermat
equation). The proof that $D'_U$ is not smooth is thus concluded with
the following lemma:
\begin{lemma} The set $S$ of  elements $p(\gamma_{b_\bullet}((b^2)^3))\in T$
generates $T$ as a vector space.
\end{lemma}
\begin{proof} Choose two independent elements $Y_0,Y_1$ of $W$. Then the
three elements $Y_0^2,\; Y_1^2,$ $ (Y_0+Y_1)^2$ are independent in $S^2W$.
For a generic choice of $a_3,\,a_4,\,a_5\in W$, the set
$$Y_0^2,Y_1^2,(Y_0+Y_1)^2, a_3^2,\,a_4^2,\,a_5^2$$ forms a basis of
$V$. We choose
$$b_\bullet=(Y_0,Y_1,Y_0+Y_1,a_3,a_4,a_5).$$
Then
\begin{eqnarray}
\label{eqgamma}
\gamma_{b_\bullet}(Y_0^2)=X_0,\,\gamma_{b_\bullet}(Y_1^2)
=X_1,
\\
\nonumber
\gamma_{b_\bullet}((Y_0+Y_1)^2)=X_2,\,\gamma_{b_\bullet}(a_i^2)=X_i,\,i=3,4,5.
\end{eqnarray}

 Choose now for $b$ a
generic linear combination of $Y_0$ and $Y_1$. Then we can write
$b^2=\alpha Y_0^2+\beta Y_1^2+\gamma (Y_0+Y_1)^2$, with the
coefficients $\alpha,\,\beta,\,\gamma$ all nonzero. It follows that
$$(b^2)^3=6 \alpha\beta\gamma (Y_0)^2(Y_1)^2(Y_0+Y_1)^2
+P((Y_0)^2,\,(Y_1)^2,\,(Y_0+Y_1)^2)$$ where the cubic polynomial $P$
contains all monomials in $(Y_0)^2,\,(Y_1)^2,\,(Y_2)^2$ containing
at least one quadratic power of one of the variables. Applying the
transformation $\gamma_{b_\bullet}$ of (\ref{eqgamma}) and the projection $p$, we get
$$p(\gamma_{b_\bullet}((b^2)^3))=6 \alpha\beta\gamma X_0X_1X_2$$
since all the monomials in the $X_j$ containing at least a quadratic
power of the variables are in $J_{f_{Fermat}}^3$. We thus proved that the
set $S$ contains $X_0X_1X_2$, and the same proof would show  that
$S$ contains $X_iX_jX_k$ for arbitrary distinct  indices
$i,\,j,\,k$. Thus $S$ generates $T$ as a vector space.

\end{proof}

The proof of Proposition \ref{proDNL} is finished.
 \end{proof}

\section{Local structure of $VSP$ for  a cubic fourfold apolar to a Veronese surface
\label{secpropsing}}

Let $W$ be a $3$-dimensional vector space, and let $g\in S^6W$, $f=s(g)\in S^3(S^2W)$ be as in the previous section, i.e. $C=V(g)$ is a plane sextic curve, and $F=V(f)$ is a cubic fourfold. Our goal in this
section is to prove the following theorem (from which Theorem
\ref{propsingintro} of the introduction immediately follows):
\begin{theorem}\label{propsing} Assume that $g$ is a general ternary sextic form, and let $f=s(g)$.

(i)     The variety $VSP(F,10)$ is smooth of dimension $4$ away from the $K3$ surface
$S_g=VSP(C,10)$. In particular, there is only one Veronese surface apolar to $f$, so we may denote by $S_f$ the
surface $S_g$.

(ii) The  singularities of $VSP(F,10)$  are quadratic nondegenerate
in the normal direction to $S_f$ at any point of $S_f$.
\end{theorem}
\begin{proof}[Proof of Theorem \ref{propsing}, (i)]
We know, by Corollary \ref{codim1},  that the set of cubics
apolar to a Veronese surface is a divisor $D_{V-ap}$  in the  space parameterizing
 all cubics.  Let
\begin{eqnarray}\label{eqvspapolar}\mathcal{V}SP_{V-ap}:=\{([Z],[f])\in
 \text{Hilb}_{10}(\PP^5)\times
 D_{V-ap},\,\,I_Z(3)\subset
 H_f\}
 \end{eqnarray}
 be the universal family of $VSP$'s of cubics apolar to a Veronese
 surface, with projection
 $$pr_2: \mathcal{V}SP_{V-ap}\to \PP(S^3V).$$  We consider the dense open $D^0_{V-ap}\subset D_{V-ap}$ defined as the set of points $[f]$
in the smooth locus of $D_{V-ap}$ such that $f=s(g)$ for a ternary sextic form $g$ of rank $10$ for which $S_g$ is integral, and Corollary \ref{corleb}
and Propositions \ref{propapdivapolar} and \ref{exampleveronese},
(i) are satisfied.

 We prove the following:

 \begin{lemma}\label{leencore}  Let $[f]\in D^0_{V-ap}$.  Then there is only one
 Veronese surface that is apolar to
 $f$, thus determining a unique curve $C$ defined by a ternary sextic form $g$ such that  $S_g=VSP(C,10)\subset VSP(F,10)$. In this case we denote by $S_f$ this surface $S_g$.

 Furthermore, denoting by $\mathcal{V}SP_{V-ap,0}$ the restriction to
 $D^0_{V-ap}$ of the family $\mathcal{V}SP_{V-ap}$,
  $\mathcal{V}SP_{V-ap,0}$ is
 nonsingular away from the family
 $\mathcal{S}\subset \mathcal{V}SP_{V-ap,0}$ of
 surfaces $S_f$.
 \end{lemma}

\begin{proof}

We
identify $f$, as before, with a hyperplane $H_f$ in $S^3V^*$.   Let
$$K\subset \text{Hom}(H_f,S^3V^*/H_f)=T_{\PP(S^3V),f})$$
be the tangent space of $D_{V-ap}$ at $[f]$, with $f=s(g)$ for some $g\in
H^0(\PP^2,\mathcal{O}_{\PP^2}(6))$ and some Veronese embedding
$\Sigma\subset\PP(V)$ of $\PP^2$. We denote again by $h\in S^3V^*$
the discriminant cubic form,  such that $V(h)$ is singular along $\Sigma$. Notice that $h\in H_f$, since $f$ is apolar to $\Sigma$.

First of all, we claim  that $K^{\perp}$ is generated by $h$. As $K$
is a hyperplane, it suffices to show that $h$ belongs to $K^{\perp}$, i.e. that  $\gamma(h)=0$ for every $\gamma\in K$.

Let $[Z]$ be a general point in $S_g\subset VSP(F,10)$. Then $Z$ is contained in a unique Veronese surface $\Sigma$ apolar to $f=s(g)$, and $\mathcal{V}SP_{V-ap}$ contains a
family
$$\mathcal{S}_U=\{[Z'],[f'])\in U|[Z']\in S_{g'}\subset VSP(F',10); f'=s(g'), F'=V(f')\}\subset \mathcal{V}SP_{V-ap}$$
 of surfaces  in a neighborhood $U$ of the point $([Z],[f])$.  Since $S_g$ is integral, we may assume that $\mathcal{S}_U$ is smooth at $([Z],[f])$.  On the other hand, the image
 $pr_2(\mathcal{S}_U)\subset \PP(S^3V)$ is dense in $D_{V-ap}$.
Now, let 
$T_{\mathcal{V}SP_{V-ap},([Z],[f])}$ be the Zariski tangent space to ${\mathcal{V}SP_{V-ap}}$ at $([Z],[f])$.   It contains  the tangent space $T_{\mathcal{S}_U,([Z],[f])}$, so since $pr_2(\mathcal{S}_U)$
is dense in $D_{V-ap}$,
the tangent space $K$ to $D_{V-ap}$ at $[f]$ is the image of the linear map
$$pr_{2*}:
T_{\mathcal{V}SP_{V-ap},([Z],[f])}\rightarrow T_{\PP(S^3V),f}.$$
So to prove the claim  it suffices to prove that the discriminant cubic form $h$ belongs to the orthogonal of
$$\text{Im}\,(pr_{2*}:
T_{\mathcal{V}SP_{V-ap},([Z],[f])}\rightarrow T_{\PP(S^3V),f}).$$ Since $g$ has rank $10$, we may assume that the scheme $Z$ consists of ten distinct points that 
impose independent conditions on cubics, so
we can identify
 $T_{\text{Hilb}_{10}(\mathbb{P}(V)),[Z]}$ with $H^0(T_{\PP^5\mid
 Z})$, and furthermore $H^0(T_{\PP^5\mid
 Z})$ with $\text{Hom}_{\mathcal{O}_{\PP(V)}}\,(\mathcal{I}_Z,\mathcal{O}_Z)$.
We have then the following description of the tangent
space of
 $\mathcal{V}SP$ at $([Z], [f]):$

 \begin{eqnarray}\label{eqtanspace2} T_{([Z],[f])}:=\{(u,\gamma)\in \text{Hom}_{\mathcal{O}_{\PP(V)}}\,(\mathcal{I}_Z,\mathcal{O}_Z)
 \times \text{Hom}(H_f,S^3V^*/H_f),
 \\
 \nonumber
 \gamma_{\mid
 I_Z(3)}=p\circ d_u: I_Z(3)\rightarrow S^3V^*/H_f\},
 \end{eqnarray}
 where $d_u: I_Z(3)\rightarrow H^0(\mathcal{O}_Z(3))$ is the map induced by $u\in \text{Hom}_{\mathcal{O}_{\PP(V)}}\,(\mathcal{I}_Z,\mathcal{O}_Z)$ on global sections, and   $p: H^0(\mathcal{O}_Z(3))\rightarrow S^3V^*/H_f$ is deduced
 from the quotient map $S^3V^*\rightarrow S^3V^*/H$,
 using the fact that the restriction map $S^3V^*\rightarrow
 H^0(\mathcal{O}_Z(3))$ is surjective and that its kernel $I_Z(3)$
 is contained in $H_f$.
We just have to prove that for $\gamma$ satisfying the equation
(\ref{eqtanspace2}), we have
\begin{eqnarray}\label{eqalpha=0} \gamma(h)=0.
\end{eqnarray}
But as $h\in I_Z(3)$, we get $\gamma(h)=d_u(h)$ modulo $H_f$, and
since $h$ is singular along $Z$, $d_u(h)=0$, which proves
(\ref{eqalpha=0}).
 The claim is thus proved.

Note that the claim   proves in particular that for $[f]\in
D_{V-ap}^0$, there is a unique Veronese surface apolar to
$f$ since it says that the cubic $h$ is determined by
$K=T_{D^0_{V-ap},[f]}$ and on the other hand it determines $\Sigma$,
because $\Sigma$ is the singular locus of $V(h)$.

 The proof of
the smoothness of $\mathcal{V}SP_{V-ap,0}$ away from $\mathcal{S}$
will now use the fact that the discriminant cubic with equation $h$
is smooth away from $\Sigma$. The argument goes as follows: Let
$[f]\in D^0_{V-ap}$, $[Z]\in VSP(F,10)$, $[Z]\not\in S_f$ and
$K=T_{D^0_{V-ap},[f]}$. Recall that the conclusion of Corollary
\ref{corleb} holds, so that $[Z]$ is a smooth point of
$\text{Hilb}_{10}(\PP(V))$. Furthermore, Proposition
\ref{propapdivapolar} also holds, so $Z$ is apolar to $f$ and
imposes independent conditions on cubics. Hence
 $I_Z(3)\subset H_f$, and this property gives us the local
equations for $VSP(F,10)$ inside $\text{Hilb}_{10}(\PP(V))_{reg}$.
Differentiating these equations, the Zariski tangent space to
$\mathcal{V}SP_{V-ap}$ at $([Z],[f])$ is thus given as before by
\begin{eqnarray}\label{eqtanspaceKK}
 T_{\mathcal{V}SP_{V-ap},([Z],[f])}:=\{(u,\alpha)\in \text{Hom}_{\mathcal{O}_{\PP(V)}}\,(\mathcal{I}_Z,\mathcal{O}_Z)
\times K,
\\
 \nonumber
 \alpha_{\mid
 I_Z(3)}=p\circ d_u: I_Z(3)\rightarrow S^3V^*/H_f\},
 \end{eqnarray}
where $K$ is the  hyperplane in $\text{Hom}(H_f,S^3V^*/H_f)$ of
linear forms vanishing on $h$.
The variety $\mathcal{V}SP_{V-ap}$ is smooth at $([Z],[f])$ if the
restriction map
$$\rho_K:K\rightarrow\text{Hom}(I_Z(3),S^3V^*/H_f)$$
 is surjective, since this implies
 that the linear equations  in (\ref{eqtanspaceKK}) defining
the Zariski tangent space to $\mathcal{V}SP_{V-ap}$ at $([Z],[f])$,
which are nothing but the differentials of the equations defining
$\mathcal{V}SP_{V-ap}$, are linearly independent.

1)  If $h$ does not  vanish identically on $Z$, then the hyperplane $K\subset \text{Hom}(H_f,S^3V^*/H_f)$ does not contain the kernel of the surjective map
$$\text{Hom}(H_f,S^3V^*/H_f)\rightarrow\text{Hom}(I_Z(3),S^3V^*/H_f)$$
so the restriction map $\rho_K$ is surjective.

2) If $h$ vanishes on $Z$, the image of the map
$K\rightarrow\text{Hom}(I_Z(3),S^3V^*/H_f)$ is the set of
linear forms on $I_Z(3)$ vanishing on $h\in I_Z(3)$.
Therefore, the linear equations in
(\ref{eqtanspaceKK}), parameterized by $\text{Hom}_{\mathcal{O}_{\PP(V)}}\,(\mathcal{I}_Z,\mathcal{O}_Z)$, are linearly dependent only if
the map
$$\text{Hom}_{\mathcal{O}_{\PP(V)}}\,(\mathcal{I}_Z,\mathcal{O}_Z)\rightarrow
\CC,$$
$$u\mapsto d_u(h)\,\,\text{mod}\,\,H_f$$
is the zero map.
 But the map \begin{eqnarray} \label{eqdu} u\mapsto
d_u(h)\in H^0(\mathcal{O}_Z(3)) \end{eqnarray} is
$H^0(\mathcal{O}_Z)$-linear. So if its image is contained in
$H_f/I_Z(3)$, it provides a sub-$H^0(\mathcal{O}_Z)$-module  of
$H^0(\mathcal{O}_Z(3))$ which is the ideal of a subscheme of $Z$
apolar to $f$. By Proposition \ref{exampleveronese}, (i), this
implies that this ideal is equal to $0$, that is, the map
(\ref{eqdu}) is $0$.

In conclusion, if $\mathcal{V}SP_{V-ap}$ is singular at $([Z],[f])$, then $Z$ is contained in the singular locus of
$h$, hence in $\Sigma$. In other words, $[Z]$ belongs to $S_f$.

\end{proof}
 Lemma  \ref{leencore}  implies (i) by a Sard type argument and this concludes the proof of
 Theorem \ref{propsing}, (i).

\end{proof}
\begin{proof}[Proof of Theorem \ref{propsing}, (ii)]
We first prove the following result:
\begin{lemma} \label{leembdim} For  general $g$ and $f=s(g)$, the
embedding dimension of $VSP(F,10)$  is $5$ at any point of $S_f$.
\end{lemma}
\begin{proof} We know that the universal family $\mathcal{VSP}$ is
smooth and that the hypersurface $\mathcal{V}SP_{V-ap}$ contains
the family of surfaces $\mathcal{S}$ which has generically smooth
fibers. If $S_f$ is smooth, the corank of the map
$pr_{2*}:T_{\mathcal{VSP},([Z],[f])}\rightarrow T_{\PP(S^3V),[f]}$ is
$1$ everywhere along $S_f$.  This implies that the
embedding dimension of $VSP(F,10)$  is $5$ at any point of $S_f$.
\end{proof}
This lemma shows that for general $g$ and $f=s(g)$, the variety
$VSP(F,10)$ has locally hypersurface singularities along $S_f$, and
our goal now is to show  that the Hessian of the local defining
equation, which is a homogeneous quadratic polynomial on the normal
bundle $N_{S_f}$, is everywhere nondegenerate. Here the bundle
$N_{S_f}$ is defined as the quotient of $T_{VSP(F,10)\mid S_f}$ by
its subbundle $T_{S_f}$. The bundle $N_{S_f}$   is thus  locally
free of rank $3$ by Lemma \ref{leembdim}.

We first have the following:
\begin{lemma}\label{ledettrivial} The determinant of $N_{S_f}$ is
trivial.
\end{lemma}
\begin{proof} We recall that by Proposition \ref{propapdivapolar}, $VSP(F,10)$ is defined as the following
set:
\begin{eqnarray} \label{eqvspdef}VSP(F,10)=\{[Z]\in
\text{Hilb}_{10}(\PP(V)),\,I_Z(3)\subset H_f\}.
\end{eqnarray}
The variety $VSP(F,10)$ is contained in the smooth part
of $\text{Hilb}_{10}(\PP(V))$ and defined according to
(\ref{eqvspdef}) as the $0$-locus of a section $\sigma$ of the bundle
$\mathcal{F}$ with fiber $I_z(3)^*$ over the point $[Z]\in
\text{Hilb}_{10}(\PP(V))$. More precisely, since we assumed that
$[f]\in D_{V-ap}^0$, the conclusion of Proposition
\ref{propapdivapolar} holds and thus $VSP(F,10)$ is contained in the
open set of $\text{Hilb}_{10}(\PP(V))$ where $\mathcal{F}$ is
locally free.  In particular, $\mathcal{V}SP\rightarrow
\PP(S^3V)$ is flat over a neighborhood of $[f]$.
 For a general $f\in S^3V$, we know by \cite{IR} that
$VSP(F,10)$ is a smooth Hyper-K\"{a}hler manifold, hence in particular
has  trivial canonical bundle.
This means that the line bundle
$$\text{det}\,(T_{\text{Hilb}_{10}(\PP^5)\mid
VSP(F,10)})\otimes (\text{det}\,\mathcal{F})^{-1}$$ has trivial
restriction to $VSP(F,10)$, which implies that it has trivial
restriction to $VSP(F,10)$ when $f$ is a general cubic apolar to a
Veronese surface, since $\mathcal{V}SP\rightarrow
\PP(S^3V)$ is flat at $[f]$.

On the other hand, the proof of Lemma \ref{leembdim} shows that the
cokernel of the differential $d\sigma$ along $S_f$ is the trivial
line bundle with fiber $\text{Hom}\,(\CC h,S^3V^*/H_f)$ at any point
$[Z]$ of $S_f$.

The exact sequence
$$0\rightarrow T_{VSP(F,10)\mid S_f}\rightarrow
T_{\text{Hilb}_{10}(\PP^5)\mid S_f} \rightarrow
\mathcal{F}_{\mid S_f}\rightarrow
\text{Coker}\,d\sigma\rightarrow0$$ thus
 implies the triviality of $\text{det}\,T_{VSP(F,10)\mid S_f} $, hence
 the triviality of $\text{det}\,N_{S_f}$
 since
$\text{det}\,T_{S_f}$ is trivial.
\end{proof}
Using the fact that the cokernel of the map $d\sigma$ is the trivial
line bundle on $S_f$, we conclude that the Hessian of $\sigma$ is a
section of $S^2N_{S_f}^*$. Here we use the following notion of Hessian for a section
$\sigma$ of a vector bundle $E$ of rank $r$ on a smooth variety $Y$, at a point $y$ where
$d\sigma$ is not of maximal rank. The Hessian is then intrinsically an
element of $({\rm Coker}\,d\sigma_y)\otimes S^2\Omega_{Y,y,\sigma}$, where
$\Omega_{Y,y,\sigma}=({\rm Ker}\,d\sigma_y)^*$. (Note that $d\sigma_y:T_{Y,y}\rightarrow E_y$ is not intrisically defined but ${\rm Ker}\,d\sigma_y$ and ${\rm Coker}\,d\sigma_y$ are.)
This Hessian is related to the usual Hessian as follows:
In an adequate local trivialization of $E$ near $y$, $\sigma$ is given by a $r$-tuple
$(\sigma_1,\ldots,\sigma_r)$
of functions on $Y$, and we can assume that if $k$ is the rank of $d\sigma$ at $y$,
then $d\sigma_1,\ldots, d\sigma_k$ are independent at the point $y$,
while $d\sigma_{k+1},\ldots, d\sigma_r$ vanish at $y$. Let $Y'$ be the smooth codimension $k$
submanifold of $Y$ defined by $\sigma_i,\,i\leq k$. Then
$\Omega_{Y,y,\sigma}=\Omega_{Y',y}$ and the restriction $\sigma_{\mid Y'}$ has zero differential
at $y$. Then the Hessian of $\sigma$ at $y$ is
the $(r-k)$-tuple of quadratic forms
$({\rm Hess}(\sigma_{k+1\mid Y'}),\ldots,{\rm Hess}(\sigma_{r\mid Y'}))$.
If furthermore we know that the vanishing locus of $\sigma$ has ordinary quadratic singularities
along a submanifold $Z\subset Y$, then near $y$, we have $Z\subset Y'$ and the Hessians
 ${\rm Hess}(\sigma_{i\mid Y'})\in S^2\Omega_{Y',y}$ appearing above  in fact belong to
$S^2N_{Z/Y'}^*$. In our case, $Z$ is $S_f$ and what we denoted by $N_{S_f}$ is naturally isomorphic
to $N_{Z/Y'}$.

As the determinant of $N_{S_f}$ is
trivial, the Hessian of  $\sigma$ as a
section of $S^2N_{S_f}^*$ is a nondegenerate quadric everywhere along $S_f$ if and
only if it is nondegenerate generically along $S_f$. The last
property can be shown as follows: Recall that $f$ is a generic cubic
apolar to a Veronese surface and $[Z]\in S_f$. The pair $([Z],[f])$ can be
constructed starting from a general subscheme of length $10$ of the
Veronese surface $\Sigma$, and taking for $H_f$ a general hyperplane
of $S^3V^*$ containing  $I_Z(3)$. Take for $Z$ a reduced subscheme
consisting of ten distinct points $x_1,\ldots,x_{10}$ in general
position on $\Sigma$. Then the hyperplane $H_f$ is determined by a linear form
$p:S^3V^*\rightarrow  S^3V^*/H_f$.   This form is the composite of the
projection map $S^3V^*\rightarrow  S^3V^*/I_Z(3)$ and a linear
form
$$p':H^0(\mathcal{O}_Z(3))\rightarrow \CC.$$
After trivialization of $\mathcal{O}_Z(3)$ we may write
$$p'=\sum_ip_i \text{ev}_{x_i}$$
for some scalars $p_i$ which can be chosen arbitrarily. Recalling
that the cokernel of $d\sigma$ is generated by
$\text{Hom}\,(\CC h,S^3V^*/H_f)$, it is clear that the Hessian
$\text{Hess}(\sigma)$ at the point $[Z]$ is obtained by restricting the sum $\sum_ip_id^2h_{x_i}$ to
$$N_{S_f,[Z]}\subset H^0(N_{\Sigma/\PP^5\mid Z})=\oplus_i
N_{\Sigma/\PP^5,x_i}.$$
 Here
we use the same trivialization  of $\mathcal{O}_Z(3)$ as above to see the
Hessian $d^2h_{x_i}$ of $h$ at $x_i$ as an element of
 $S^2N_{\Sigma/\PP^5,x_i}^*$.
 Since $h$ has nondegenerate quadratic singularities along $\Sigma$,
 each of the quadrics
$d^2h_{x_i}$ is nondegenerate. We now have:
\begin{lemma}\label{leorthog} The $3$-dimensional vector space $N_{S_f,[Z]}$ is
the orthogonal complement of
the subspace $\text{Im}\,(H^0(\Sigma,N_{\Sigma/\PP^5})\rightarrow
\oplus_iN_{\Sigma/\PP^5,x_i})$ with respect to the quadratic form $\sum_ip_id^2h_{x_i}$.
\end{lemma}
\begin{proof} Indeed, the space
$N_{S_f,[Z]}$ is equal to the kernel of the composite map
$$H^0(N_{\Sigma/\PP^5\mid Z})\rightarrow
\text{Hom}\,(I_\Sigma(3),H^0(\mathcal{O}_Z(3))\stackrel{p'}{\rightarrow}
\text{Hom}\,(I_\Sigma(3),S^3V^*/H_f),$$ where $ H^0(N_{\Sigma/\PP^5\mid
Z})\cong \oplus_iN_{\Sigma/\PP^5,x_i}$ and $p'=\sum_ip_i\text{ev}_{x_i}$.

 Let now $u\in H^0(\Sigma, N_{\Sigma/\PP^5}), u_{\mid Z}=(u_i)$ and $v=(v_i)\in H^0(
N_{\Sigma/\PP^5\mid Z})$. Then
$$(\sum_ip_id^2h_{x_i})(u_{\mid Z},v_{\mid Z})=\sum_ip_i
d^2h_{x_i}(u_{i},v_{i}).$$ The section $u\in H^0(\Sigma, N_{\Sigma/\PP^5})$
lifts to a section $U\in H^0(\PP^5,T_{\PP^5})$.  Let $d_U:S^3V^*\to S^3V^*$ be the induced map on cubic forms. Then
the degree $3$ polynomial $d_U(h)$ belongs to $I_\Sigma(3)$.
Furthermore we have
$$d^2h_{x_i}(u_{i},v_{i})=d(d_U(h))(v_i)$$ for any $i$. It follows
that
$$\sum_ip_i
d^2h_{x_i}(u_{i},v_{i})=\sum_ip_id(d_U(h))(v_i).$$ If now
 $(v_i)$
belongs to $N_{S_f,[Z]}$, we find that $\sum_ip_id(d_U(h))(v_i)=0$ and
thus $$\sum_ip_i d^2h_{x_i}(u_{i},v_{i})=0.$$ Hence we proved that
$\text{Im}\,(H^0(\Sigma,N_{\Sigma/\PP^5})\rightarrow
\oplus_iN_{\Sigma/\PP^5,x_i})$ is perpendicular with respect to
$\sum_ip_id^2h_{x_i}$ to the space $N_{S_f,[Z]}$. As the space
$H^0(\Sigma,N_{\Sigma/\PP^5})$ is of dimension $27$, the map
$H^0(\Sigma,N_{\Sigma/\PP^5})\rightarrow \oplus_iN_{{\Sigma/\PP^5},x_i}$ is
injective of maximal rank $27$ for a general choice of the $x_i$'s.
As  the space $N_{S_f,[Z]}$ is of dimension $3$,  we conclude that
$$\text{Im}\,(H^0(\Sigma,N_{\Sigma/\PP^5})\rightarrow
\oplus_iN_{{\Sigma/\PP^5},x_i})$$
 is the orthogonal complement with respect
to $\sum_ip_id^2h_{x_i}$ of  the space $N_{S_f,[Z]}$, since the quadratic
form $\sum_ip_id^2h_{x_i}$ on the $30$-dimensional vector space
$\oplus_iN_{{\Sigma/\PP^5},x_i}$ is nondegenerate.
\end{proof}

It follows that the quadratic form $\text{Hess}(\sigma)$, that is the
restriction of $\sum_ip_id^2h_{x_i}$ to $N_{S_f,[Z]}$, is nondegenerate
if and only if the quadratic form  $\sum_ip_id^2h_{x_i}$ has a
nondegenerate restriction to
$\text{Im}\,(H^0(\Sigma,N_{\Sigma/\PP^5})\rightarrow
\oplus_iN_{{\Sigma/\PP^5},x_i})$. The last property may be achieved   because  the
points $x_i$ being general, the map
$$H^0(\Sigma,N_{\Sigma/\PP^5})\rightarrow
\oplus_{1\leq i\leq 9}N_{{\Sigma/\PP^5},x_i}$$ is injective (hence an
isomorphism). Hence any combination $\sum_{1\leq i\leq
9}p_id^2h_{x_i}$ with $p_i\not=0$ for any $i\leq9$ has a
nondegenerate restriction to
$\text{Im}\,(H^0(\Sigma,N_{\Sigma/\PP^5})\rightarrow \oplus_iN_{{\Sigma/\PP^5},x_i})$
and thus  a general  combination $\sum_{1\leq i\leq 10}p_id^2h_{x_i}$
 has a nondegenerate restriction to
$\text{Im}\,(H^0(\Sigma,N_{\Sigma/\PP^5})\rightarrow
\oplus_iN_{{\Sigma/\PP^5},x_i})$.

In conclusion, we proved that, for general $g$ and $f=s(g)$, at a general point
$[Z]\in S_f=S_g\subset VSP(F,10)$, the Hessian of the local defining
equation of $VSP(F,10)$ has rank $3$, and as explained above, this
implies that it is everywhere nondegenerate in the  normal direction
to $S_f$.
\end{proof}
\section{Proof of Theorem \ref{main}}
We first recall the statement of the result:
\begin{theorem} \label{theofin}Let $F$ be a very general cubic fourfold. Then there is
no nonzero morphism of Hodge structures between
$H^4(F,\QQ)_{prim}$ and $H^2(VSP(F,10),\QQ)_{prim}$.
\end{theorem}
\begin{proof} Let $B$ be the Zariski open set of
$\PP(H^0(\PP^5,\mathcal{O}_{\PP^5}(3)))$ parameterizing smooth
cubics. We have the universal family $\pi: \mathcal{X}\rightarrow B$
of cubic hypersurfaces, where the morphism $\pi$ is smooth and
projective. We also have the family $\pi':\mathcal{V}SP\rightarrow
B$ which is projective over $B$ but is not smooth. By Proposition \ref{exampleveronese} the general cubic fourfold apolar to a Veronese surface is smooth, so the base $B$
contains the divisor $D_{V-ap}$ parameterizing smooth cubic fourfolds
apolar to a Veronese surface. We proved in Theorem \ref{propsing}
that for $[f]$ in an open subset $ D_{V-ap}^0$, the fiber
$VSP(F,10)={\pi'}^{-1}([f])$ has only ordinary
quadratic singularities along the surface $S_f$ which is a smooth
$K3$ surface. Let $[f]$ be a point of $D_{V-ap}^0$ and let $B^0$ be a
Zariski open set of $B$ containing $[f]$ and such that $D_{V-ap}\cap
B^0\subset D^0_{V-ap}$. Let $B'\rightarrow B^0$ be the double cover
ramified along $D^0_{V-ap}$. Since $D^0_{V-ap}$ is contained in the smooth locus of $D_{V-ap}$ (cf. Lemma \ref{leencore}), the double cover, $B'$, is smooth, and the
pulled-back family $\tilde{\pi}':\mathcal{V}SP'\rightarrow B'$ is
smooth except along the family of surfaces $\mathcal{S}\rightarrow
D^0_{V-ap}$, which has codimension $3$ in $\mathcal{V}SP'$, and
along which $\mathcal{V}SP'$ has  quadratic nondegenerate
singularities. The family $\mathcal{V}SP'\rightarrow B'$ can be
modified after passing to a degree $2$ \'{e}tale cover of $B'$ to a
family of smooth complex projective manifolds  by a small
resolution: For this we first blow-up $\mathcal{V}SP'$ along
$\mathcal{S}$ to get $\mathcal{V}SP''\rightarrow B$. The exceptional
divisor $E$ of the blow-up is a bundle over $\mathcal{S}$ with
fibers smooth two-dimensional quadrics. There is an \'{e}tale double
cover $\widetilde{\mathcal{S}}\rightarrow \mathcal{S}$
parameterizing the rulings in the fibers of $E\rightarrow
\mathcal{S}$. As a $K3$ surface is simply connected, this double
cover comes from a double cover $\widetilde{D_{V-ap}^0}\rightarrow
D_{V-ap}^0$. We may assume this \'etale double cover is induced by
an \'etale double cover $\widetilde{B}^0\rightarrow B^0$. Performing
this base change, the pulled-back family
$\widetilde{\mathcal{V}SP''}\rightarrow \widetilde{B}^0$ has the
property that the inverse image $\widetilde{E}$ of $E$ admits two
morphisms to a $\PP^1$-bundle over $\widetilde{\mathcal{S}}$. We
choose one of them, and as is well-known, we can contract
$\widetilde{E}$ to $\widetilde{\mathcal{S}}$ along this morphism.
The resulting family $\phi:\widetilde{\mathcal{V}SP}\rightarrow
\widetilde{B}^0$ is smooth proper over $\widetilde{B}^0$.

We now have two families $$\phi:\widetilde{\mathcal{V}SP}\rightarrow
\widetilde{B},\,\,\psi:\widetilde{\mathcal{X}}\rightarrow \widetilde{B}$$ of
smooth proper complex  manifolds, where
$\widetilde{\mathcal{X}}:=\mathcal{X}\times_{B}\widetilde{B}^0$. The fibers of
both families are projective, and in particular K\"{a}hler, although it
is not clear if both morphisms are projective.
We thus get two associated variations of Hodge structures on
$\widetilde{B}$, one of weight $2$ on the primitive cohomology of
degree $2$ of the fibers of the first family with associated local
system $H^2$, the other of weight $4$ on the primitive cohomology of
degree $4$ of the fibers of the second family with associated local
system $H^4$. The locus of points $b\in\widetilde{B}$ where there is
a nonzero morphism of Hodge structures
$H^4(\widetilde{\mathcal{X}}_b,\QQ)_{prim}\rightarrow
H^2(\widetilde{\mathcal{V}SP}_b,\QQ)_{prim}$ is the Hodge
locus for the induced variation of Hodge structure on the local
system $\text{Hom}\,(H^4,H^2)$. The Hodge locus is a countable union
of closed algebraic subsets of the base $\widetilde{B}$ (cf.
\cite{voisinhodgeloci}). In order to prove Theorem \ref{theofin}, it
thus suffices to prove that there is a point of $\widetilde{B}$
where there is no nonzero morphism of Hodge structures between
$H^4(\widetilde{\mathcal{X}}_b,\QQ)_{prim}$ and
$H^2(\widetilde{\mathcal{V}SP}_b,\QQ)_{prim}$.

By Proposition \ref{propnonNL},  the divisor $D_{V-ap}$ is not a
Noether-Lefschetz locus for the family $\mathcal{X}\rightarrow B$.
This means that there exists a point $b\in D_{V-ap}$, that we may assume to
be in $D^0_{V-ap}$, such that there is no nonzero  Hodge class in
$H^4(\mathcal{X}_b,\QQ)_{prim}$. This fact implies that the
Hodge structure on $H^4(\mathcal{X}_b,\QQ)_{prim}$ is simple.
Indeed, since $h^{3,1}(\mathcal{X}_b)=1$, any proper sub-Hodge
structure has $h^{3,1}$-number $0$ or its orthogonal complement for
the intersection pairing satisfies this property. In both cases, the
existence of a proper sub-Hodge structure implies the existence of a
nonzero Hodge class. Note also that it has $h^{2,2}$-number  equal
to $20$.

On the other hand, we claim that the transcendental part of
$H^2(\widetilde{\mathcal{V}SP}_b,\QQ)_{prim}$ has
$h^{1,1}$-number $\leq 19$. Here the transcendental part is defined
as the minimal sub-Hodge structure containing the
$H^{2,0}$-component.

The claim follows from the fact that $\widetilde{\mathcal{V}SP}_b$
is hyper-K\"{a}hler, being a fiber of a family of K\"{a}hler manifolds whose
general member is hyper-K\"{a}hler, and on the other hand it  is the
blow-up of $\mathcal{V}SP_b$ along the $K3$ surface $S_b$. It thus
contains the exceptional divisor $E_b$ over $S_b$ and the morphism
of Hodge structures
$H^2(\widetilde{\mathcal{V}SP}_b,\QQ)\rightarrow
H^2(E_b,\QQ)$ does not vanish on
$H^{2,0}(\widetilde{\mathcal{V}SP}_b)$ because a symplectic form on
a fourfold cannot vanish on a divisor. On the other hand, this
morphism sends $H^2(\widetilde{\mathcal{V}SP}_b,\QQ)_{tr}$ to
$H^2(E_b,\QQ)_{tr}$ which is equal to
$H^2(S_b,\QQ)_{tr}$. The induced morphism
$$H^2(\widetilde{\mathcal{V}SP}_b,\QQ)_{tr}
\rightarrow H^2(S_b,\QQ)_{tr}$$ must be injective by the same
simplicity argument  as above, and thus
$$h^{1,1}(\widetilde{\mathcal{V}SP}_b,\QQ)_{prim}\leq
h^{1,1}(S_b)_{prim}\leq19.$$
As the Hodge structure on $H^4(\mathcal{X}_b,\QQ)_{prim}$ is simple with
$h^{2,2}$-number equal to $20$, any morphism of Hodge structures
between $H^4(\mathcal{X}_b,\QQ)_{prim}$ and a weight $2$ Hodge structure with
$h^{1,1}$-number $\leq 19$ is identically $0$, which
concludes the proof of Theorem \ref{main}.

\end{proof}

  \end{document}